\documentclass[11 pt]{article}   
\usepackage[utf8]{inputenc}
\usepackage{authblk}
\usepackage{qsymbols,graphicx,amsfonts,fullpage,enumitem, comment,amsthm,color,amsmath, todonotes,ulem,mdframed,hyperref,cleveref}

\def \ben{\begin{eqnarray}}
\def \be{\begin{eqnarray*}}
\def \be{\begin{eqnarray*}}
\def \een{\end{eqnarray}}
\def \ee{\end{eqnarray*}}
\def \beq{\begin{equation}}
\def \eq{\end{equation}}
\def \bpar#1{\left\{\begin{array}{#1} }
\def \epar { \end{array}\right.}
\def\l{\left}
\def\r{\right}
\def \1{\textbf{1}}

\def \sur#1#2{\mathrel{\mathop{\kern 0pt#1}\limits^{#2}}}
\def \eqd{\sur{=}{(d)}}
\def \bls{{\tiny $\blacksquare$}}
\def \bma{\begin{bmatrix}}
\def \ema{\end{bmatrix}}
\def \E{\mathsf{E}}
\def \P{\mathsf{P}}
\def \Cov{\mathsf{Cov}}
\def \Var{\mathsf{Var}}
\def \BNEG{{\sf BNEG}}
\def \bpar#1{\left\{\begin{array}{#1} }
\def \epar { \end{array}\right.}
\def \R{\mathbb{R}}
\def \eref#1{(\ref{#1})}
\def \floor#1{\lfloor#1\rfloor}

\def \se{{\sf e}}
\def \cro#1{\llbracket#1\rrbracket}
\def \bT{\mathbf{T}}
\def\U{\mathbb{U}}
\newtheorem{lemma}{Lemma}
\newtheorem{theorem}[lemma]{Theorem}

\newtheorem{proposition}[lemma]{Proposition}

\newtheorem{conjecture}[lemma]{Conjecture}

\newtheorem{definition}[lemma]{Definition}
\newtheorem{definition-lemma}[lemma]{Definition-Lemma}
\newtheorem{remark}[lemma]{Remark}

\def \Z{\mathbb{Z}}
\def \Bkn{{\cal B}_{k}^n}
\def \tkn#1{t_{#1}^n}
\def \sigmakn#1{\sigma_{#1}^n{}}
\def \xikn#1{\xi_{#1}^{n,\star}}
\title{Note on the density of ISE and a related diffusion}

\author[$\dagger$]{Guillaume Chapuy}
\author[$\ddag$]{Jean-Fran\c{c}ois Marckert}
\affil[$\dagger$]{\small Universit\'e Paris Cit\'e, CNRS, IRIF, F-75013, Paris, France}
\affil[$\ddag$]{\small 
Univ. Bordeaux, CNRS, Bordeaux INP, LaBRI, UMR 5800, F-33400 Talence, France}
\date{}
\begin{document}
\normalem
\maketitle

\begin{abstract}
    The integrated super-Brownian excursion (ISE) is the occupation measure of the spatial component of the head of the Brownian snake with lifetime process the normalized Brownian excursion. It is a random probability measure on $\R$, and it is known to describe the continuum limit of the distribution of labels in various models of random discrete labelled trees.
	
We show that $f_{ISE}$, its (random) density  has a.s. a derivative $f'_{ISE}$ which is continuous and $\left(\frac{1}{2}-`e\right)$-H\"older for any $`e >0$ but for no $`e<0$ (proving a conjecture of Bousquet-Mélou and Janson).  
	We conjecture that  $f_{ISE}$ can be represented as a second-order diffusion of the form 
	$$df'_{ISE}(t) = \sqrt{2f_{ISE}(t)}\, dB_t + g\l(f'_{ISE}(t), f_{ISE}(t),\int_{-\infty}^t f_{ISE}(s)ds\r)dt,$$ for some continuous function $g$,  for $t>0$, and we give a number of remarks and questions in that direction.

	The proof of regularity is based on a moment estimate coming from a discrete model of trees, while the heuristic of the diffusion comes from an analogous statement in the discrete setting, which is a reformulation of explicit product formulas of Bousquet-M\'elou and the first author (2012).
\end{abstract}

\begin{figure}
  \centerline{\includegraphics[width=55mm]{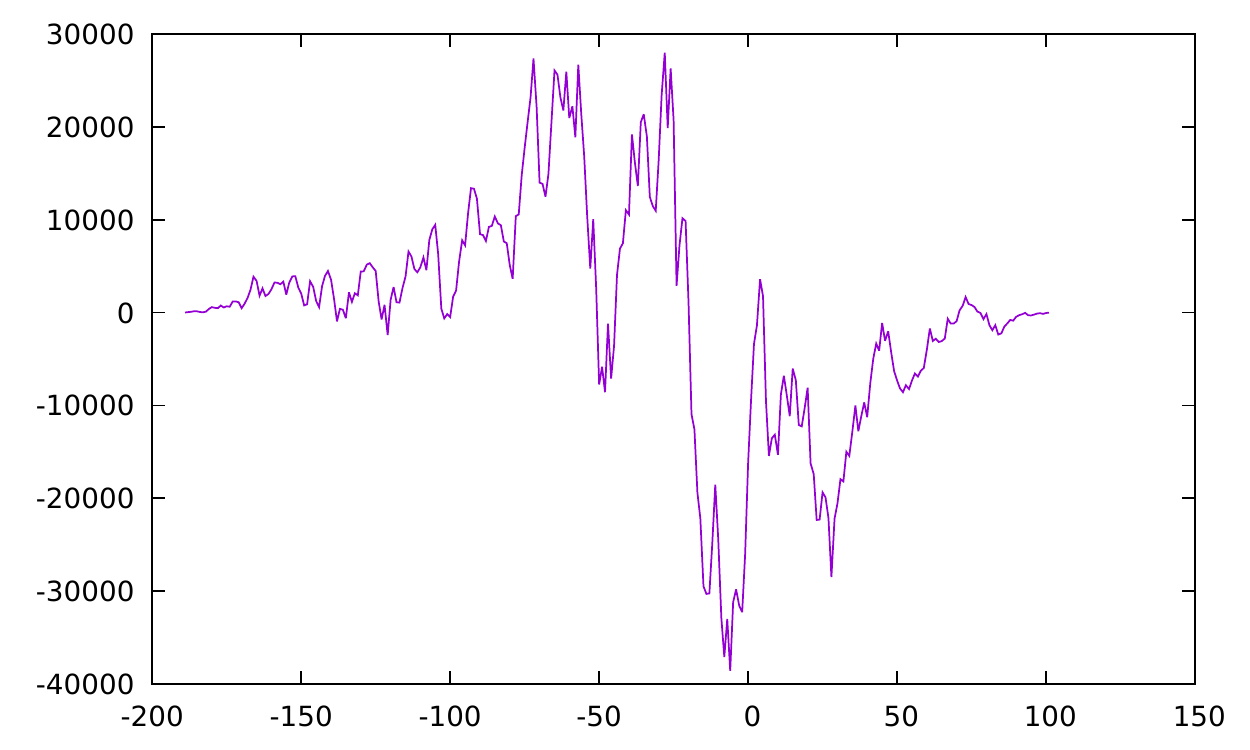}~~\includegraphics[width=55mm]{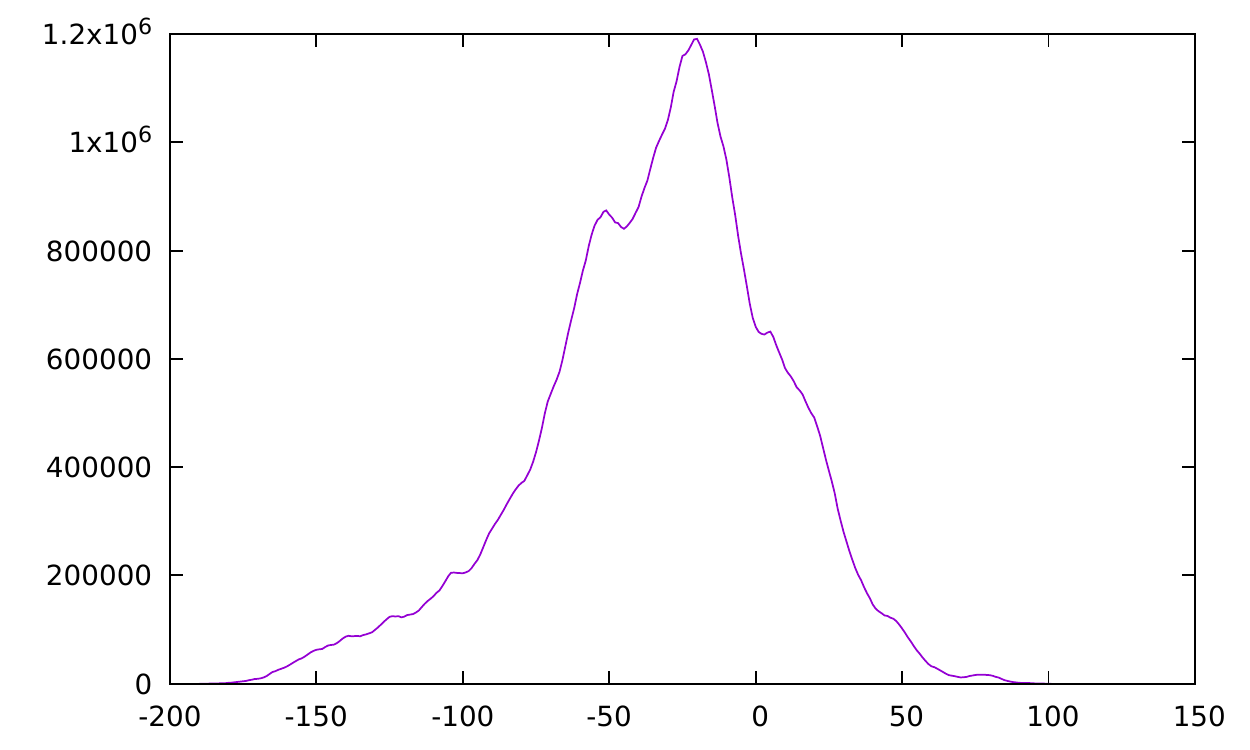} ~~\includegraphics[width=55mm]{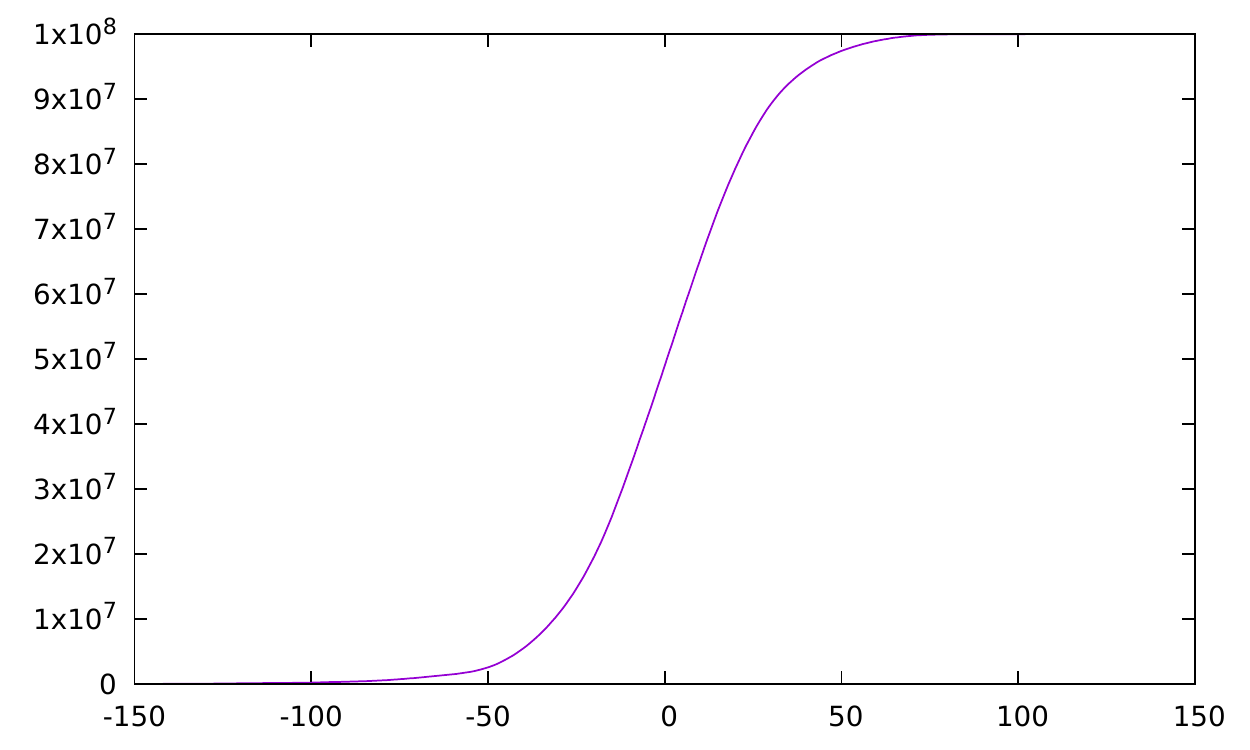}}
   \caption{\label{fig:illus} Uniform planar tree with 50 millions of vertices, spatial increments i.i.d. $\pm 1$ with probability $1/2$. From left to right, representation of the processs $\Delta(T)$, $M(T)$, $S(T)$. They are discrete approximations of the continuum processes $f'_{ISE}(t),f_{ISE}(t), \int_{-\infty}^t f_{ISE}(s)ds$.}
  \end{figure}
  
\section{Introduction}

\subsection{The ISE}

In this note we are interested in a random probability measure called the Integrated SuperBrownian Excursion, or ISE. 
We give here its most intrinsic definition via the Brownian snake. Our combinatorially inclined readers may prefer to think about it in terms of continuum limits of random trees (for this, see Proposition~\ref{prop:BMJ} below, which can be taken as a definition).

\smallskip

Recall that a Brownian snake $(W,c)=((W_s,c_s),s\in[0,1])$ with lifetime process $c=2\se$, where $\se$ is the normalized Brownian excursion, is a family of Brownian motions indexed by a continuum random tree with contour process $c$, namely:
\begin{itemize}[itemsep=0pt, topsep=0pt,parsep=0pt, leftmargin=24pt]
\item[--] For each $s\in[0,1]$, the spatial component $W_s$ is a Brownian motion indexed by $[0,c_s]$,
\item[--] for $0\leq s\leq t\leq 1$, $W_s$ and $W_t$ coincide on $[0,T_{s,t}]$ for $T_{s,t}=\min\{c_u, s\leq u \leq t\}$,
\item[--]the two Brownian motions $(W_s(T_{s,t}+u)-W_s(T_{s,t}), 0\leq u \leq c_s-T)$ and $(W_t(T_{s,t}+u)-W_t(T_{s,t}), 0\leq u\leq c_t-T_{s,t})$ are independent.
\end{itemize}
\noindent 
The lifetime process $2\se$ encodes a random tree $T_{2\se}$, usually called Aldous' continuum random tree, or CRT. The Brownian snake can be viewed as a Brownian motion indexed by $T_{2\se}$. Each real number $s$ in $[0,1]$ encodes a node in the tree, at depth $2\se(s)$, and the Brownian motion $W_s$ is the spatial component function: for $t\in[0,2\se(s)]$, $W_s(t)$ is simply the spatial position of the node of the branch from the root to the node $s$, which is at depth $t$.

The head of the (spatial component of the) Brownian snake is the process  $H=(H_s,s\in[0,1])$ of terminal points of the Brownian snake, and it is defined on $[0,1]$ by $H_s:=W_s(c_s)$. Conditional on $c$, it is a centred Gaussian process with covariance function  $\Cov(H_s,H_t)= \min\{c_u, u \in[s \wedge t, s\vee t]\}$ (for $0\leq s,t \leq 1$). 

The \emph{Integrated SuperBrownian Excursion}, or \emph{ISE}, is the random probability measure $\mu_{ISE}$ defined as the occupation measure of $H$: for all continuous function $g:\R\to\R$ with compact support
\[ \int_{-\infty}^{+\infty} g\, d\mu_{ISE} = \int_0^1 g(H_s) \,ds.\]
We refer to Le Gall \cite{LG2005}, or \cite{J-M,M-M} for more information on the subject.

It is known that $\mu_{ISE}$ has almost surely (a.s.) a (random) compact support~: this is a consequence of the a.s. continuity of $H$ (see~\cite{BM} and \cite{DEL} for the distribution of the support). Moreover, as shown by Bousquet-Mélou and Janson \cite{BM-J}, $\mu_{ISE}$ is a.s. absolutely continuous with respect to the Lebesgue measure and its random density has a continuous version denoted $f_{ISE}$: the study of the random process $f_{ISE}$ is the aim of this paper.

 \subsection{Main contributions of this note.}

The main point of this note is to convey the intuition that not only $\mu_{ISE}$ has a continuous density $f_{ISE}$, but this density has a continuous derivative $f'_{ISE}$, and we expect $f'_{ISE}$ to behave, in some sense, as a diffusion. In particular, we expect the 3D-process made by $f_{ISE}$, its derivative $f'_{ISE}$, and its cumulative integral  $t\mapsto\int_{-\infty}^t f_{ISE}(s)ds$, to be Markovian -- which as far as we know has never been suggested before.\medskip

\noindent Technically, our main contributions are the following: 
 \begin{itemize}[itemsep=0pt, topsep=0pt,parsep=0pt, leftmargin=24pt]
\item[\bls]  We prove (Theorem~\ref{thm:tightness} in Section \ref{sec:tight}) that indeed $f'_{ISE}$ exists and is continuous and even $\left(\frac{1}{2}-`e\right)$-H\"older for any $`e >0$. However, this is not true for any $`e<0$ (Proposition~\ref{prop:bonus} in Section~\ref{sec:gaussianPairs}), and in particular $f_{ISE}$ has no second derivative. This proves the main conjecture in Bousquet-Mélou and Janson~\cite[Conjecture 2.3]{BM-J}.  
\item[\bls] We introduce a 3D-valued process, noted $\zeta$, made by $f_{ISE}$, its derivative , and its integral.  The main message of this note is that this 3D-process opens the way to an understanding of $ISE$ in a Markovian way.
\item[\bls] We observe (Proposition~\ref{prop:discreteDiff} in Section~\ref{sec:comp}) that, in the discrete model of uniform random binary trees, whose convergence to ISE is known, the discrete analogue of the triple process $\zeta$ can be represented as a Markov process, when conditioned to its values at the \emph{two} boundaries of a discrete interval. This is a direct consequence,  which seems to have been unnoticed before, of a theorem of~\cite{BM-C}.
\item[\bls] We prove that this discrete Markov process indeed converges to a diffusion (Proposition~\ref{thm:diffApprox} in Section \ref{sec:diffApprox}), when conditioned to its value at only \emph{one} boundary of a given interval, and properly tamed near $0$.
\item[\bls] We conjecture (Section~\ref{seq:sq}) that $f'_{ISE}$ (more precisely, $\zeta$) can be represented by a diffusion, with or without conditioning at boundaries. We give a number of questions related to this.
\end{itemize}
We hope to stimulate efforts by researchers best acquainted with the subject, with the hope that someone can give a rigorous meaning (or several) to this statement, and prove it.
 
\medskip

\paragraph{Acknowledgement.} The authors have kept the idea of this project in mind for several years, without making much progress on the technical aspects. The first author talked to Jean-François Le Gall in CIRM in December 2021, who seemed quickly able to make contributions to the subject. A paper by Le Gall related to some of the questions we ask here is in preparation, and this is in fact the main motivation for us to finalize this note in its present form. We thank Jean-François Le Gall for his interest on the problem and for his encouragement to write this note, and we hope that other experts of the subject will be interested in these questions.

\section{Convergence results and related questions}

\subsection{Discrete approach to $f_{ISE}$.} From the discrete perspective, the ISE is the limit in distribution of several natural random probability measures, in particular the distribution profile of distances to a random vertex in a random planar map~\cite{CS}, and most directly, of the distribution of labels in various models of random spatial trees~\cite{Aldous_m,M-M,J-M,M12, D-J}.
The model of "spatial trees" involved are discrete branching random walks (sometimes viewed as discrete snakes): some random values are attached to the edges of a discrete random tree, and the label/abscissa attached with a node $u$, is obtained by summing the values on the path going from the root of the tree to $u$. These are discrete analogues to the Brownian snake (indexed by the CRT).
In order that the abscissa process along long branches converge a Brownian motion, the random values are taken to be centred, with finite variance (in fact, fourth moments are needed in the case of trees converging to the continuum random tree, after distance normalization by $\sqrt{n}$); the random values are also assumed to be independent, but dependence between values associated to "children-edge" of the same node, as well as a relaxation of the mean 0 condition replaced by a global centring condition (under random permutations of the sibling of vertices) is allowed (see \cite{M12}), and this encompasses the discrete models we will need here.

Hence, to study ISE, a possible way consists in choosing a simple discrete model having the Brownian snake as a limit (or simply, having a spatial occupation measure converging to ISE).

\begin{figure}[t]
	\begin{center}
		\includegraphics[width=0.6\linewidth]{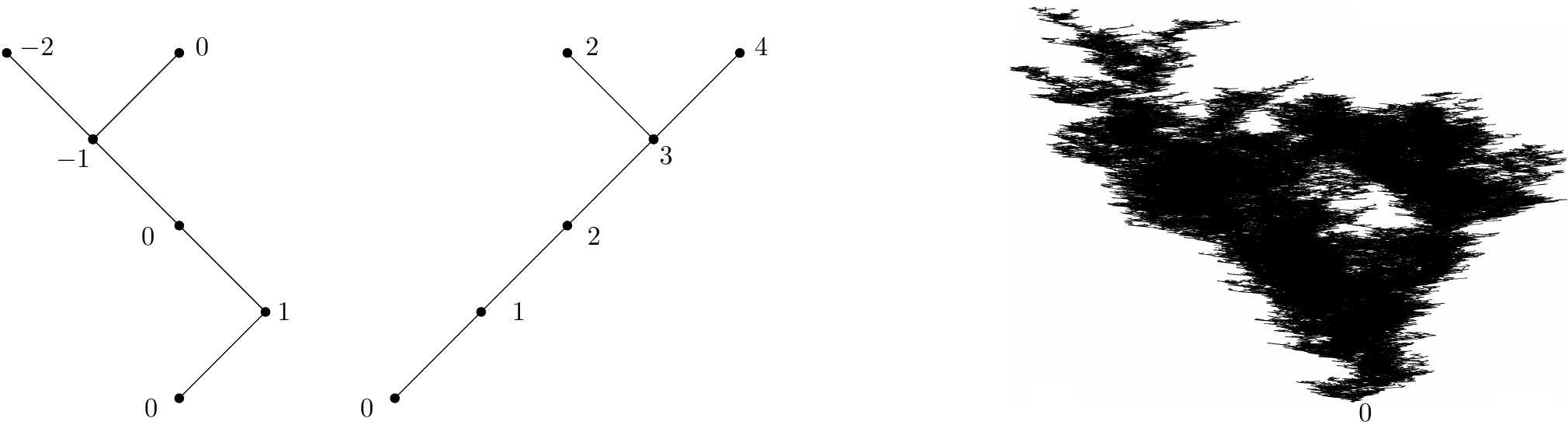}
	\end{center}
	\caption{Left: Two different binary trees with $n=6$ vertices. The abscissa of each vertex is displayed next to it. These two trees have respective vertical profile $(1,1;3,1)$ and $(;1,1,2,1,1)$. Right: A uniformly random binary tree $T$ of size $n$ with $n\approx 10^5$; the projection of the uniform measure of points on the horizontal axis is the measure $\mu_T(n)$ -- and, properly normalized, it is a good approximation of $\mu_{ISE}$.}
	\label{fig:binarytree}
\end{figure}

For the purpose of this note, we will use the fact that the ISE is the weak limit of a discrete model for which some useful formula have been obtained (\cite{BM-C}).

\paragraph{Binary trees}
For us a binary tree is a rooted plane tree in which each vertex has a (possibly empty) \emph{left} subtree and a (possibly empty) \emph{right} subtree. See Figure~\ref{fig:binarytree}. 
\begin{remark}
    Some authors would say an \emph{incomplete} binary tree. Replacing each empty subtree by a single leaf gives a \emph{complete} binary tree (plane tree with only vertices of arity $2$ and $0$). 
    \end{remark}

The number of binary trees with $n$ vertices is the $n$-th Catalan number, $Cat(n)=\frac{1}{n+1}{2n \choose n}$. We define the \emph{abscissa}~\cite{BM, BM-C} of a vertex $v$ as the number of right steps minus the number of left steps on the path from the root to $v$. 
The uniform distribution on the set ${\cal B}_n$ of binary trees with $n$ vertices is denoted $\U_n$. A binary tree equipped with the labelling of its vertices by their abscissas, gives a model of spatial tree whose distribution of labels, as we will see, converges after an appropriate rescaling to ISE.

The \emph{vertical profile} of the tree $T$ is the tuple
\[M(T)=\l(M_{\ell(T)}(T),\cdots,M_{-1}(T); M_0(T),\cdots,M_{r(T)}(T)\r)\]
where $\ell=\ell(T), r=r(T)$ are respectively the smallest  and largest abscissa ($\ell$ and $r$ stand for left and right), and where for each integer $i$, $M_i(T)$ is the number of vertices of $T$ of abscissa $i$.
For probabilistic applications, it is most natural to turn this object into a probability measure. Introducing a suitable normalization, we thus define the probability measure
\begin{align}\label{def:profile}
	\mu_T(n) := \frac{1}{n} \sum_{k=\ell(T)}^{r(T)} M_k(T)\, {\sf Dirac}_{  k / (2n)^{1/4} },
\end{align}
where ${\sf Dirac}_x$ is the Dirac mass at $x$.
Note that this carries all the information of the profile, and in fact, it is the measure $\mu_T(n)$ that we will refer to as \emph{the profile of $T$}. In addition to $M(T)$, we introduce two processes 
$\Delta(T):=(\Delta_i(T), i \in \mathbb{Z})$ and
$S(T):=(S_i(T), i \in \mathbb{Z})$,  
which can be thought of as the discrete derivative and discrete integral of $M(T)$, respectively:
\[\Delta_i(T)=M_i(T)-M_{i-1}(T),~~~ S_i(T)=\sum_{j \leq i} M_j(T), ~~~~~ i\in \mathbb{Z}.\]
Finally,  we define the triple
\[ Z_i(T)=(\Delta_i(T),M_i(T),S_i(T)), ~~~~~ i\in \mathbb{Z}.\]
Note that each of $\Delta(T)$, $M(T)$ and $S(T)$ determines $Z(T):=(\Delta(T),M(T),S(T))$, and observe also that the knowledge of $(\Delta_i, M_i)$ is equivalent to that of  $(M_{i-1},M_i)$.  \par
The following Proposition, due to Bousquet-Mélou \& Janson \cite{BM-J}, states that $\mu_{ISE}$ is nothing but the limit in distribution of the profile of a random element of $\U_n$.
\begin{proposition}[{\cite[Theo. 3.1]{BM-J}}]\label{prop:BMJ}
	If $\bT_n$ is picked according to $\U_n$, then $\mu_{\bT_n}(n)$ converges in distribution to $\mu_{ISE}$, for the topology of  weak convergence of probability measures. Moreover, $(\frac{1}n(2n)^{1/4}M_{(2n)^{1/4}x}(\bT_n),x\in \R)$ converges in distribution to $f_{ISE}$ in $C_0$, the space of continuous functions with 0 limit at $\pm \infty$, equipped with the topology of uniform convergence.
\end{proposition}
In the proposition, $M(\bT_n)$ is the continuous process which interpolates $(M_k(\bT_n),k \in \mathbb{Z})$ linearly between integer points. The same notation will be used for all integer-indexed processes encountered in the paper.\par
This result implies that $f_{ISE}$ is almost surely continuous (which can be stated also as: $\mu_{ISE}$ possesses a.s. a density $g_{ISE}$ with respect to the Lebesgue measure on $\R$, and the process $g_{ISE}$ admits a continuous version $f_{ISE}$). These authors also show that $f_{ISE}$ has a derivative a.e. $f'_{ISE}$, and they conjecture that $f'_{ISE}$ possesses a continuous version, but no second derivative (see Fig. \ref{fig:illus} for a simulation of the processes $\Delta,M,S$ of a large random tree, on which one can guess, among other things, the regularity of the limiting processes).

The first assertion of the proposition can also be seen as a consequence of~\cite{M12}, using the fact that the model of (spatial) trees under $\U_n$ can be realized as globally centred discrete snakes (take the following offspring distribution $p_0=1/4$, $p_1=1/2$, $p_2=1/4$. This offspring distribution is critical, and the corresponding variance $\sigma^2=1/2$. In the case where a node $u$ has one child, $v$, the spatial displacement associated with the edge $(u,v)$ is $+1$ or $-1$ with probability $1/2$, if there are two children, the spatial displacement is $(-1)$ for the left child, $+1$ for the right one).

For the interested reader, we point out that Devroye \& Janson \cite{D-J} obtained analogues of Proposition~\ref{prop:BMJ} for general models of discrete snakes
under the hypothesis that the spatial increments along edges are independent (which is not the case here). 

\subsection{Tightness  and regularity}
\label{sec:tight}
Let $\bT_n$ be a random binary tree taken under distribution $\U_n$. For $k\in \mathbb{Z}$, we let $M_k^n=M_k(\bT_n)$, $\ell_n=\ell(\bT_n), r_n=r(\bT_n)$, $\Delta^n=\Delta(\bT_n)$ and $S^n=S(\bT_n)$ be the discrete derivative and integral of this process.
We introduce the three-dimensional process $\zeta^{n}=(\zeta^{n}(t), t \in \mathbb{R})$ defined  by,
$$\zeta^{n}(t) :=
 \left(
	 \frac{1}{n^{1/2}}\Delta_{n^{1/4}t}^{n} , 
	 \frac{1}{n^{3/4}} M_{n^{1/4}t}^{n}, 
	 \frac{1}{n}       S_{n^{1/4}t}^{n}
 \right)_{t\in \mathbb{R}},
$$
where we recall that our notation uses implicitely linear interpolation between integer values of~$n^{1/4}t$.
The choice of normalization for the three coordinates will become clear below.
\begin{theorem}\label{thm:tightness}
	The process $\zeta^{n}$ is tight in $\mathcal{C}_K(\mathbb{R})^2\times C_{0,1}$, where $C_K(\R)$ is the space of continuous functions with compact real support, and $C_{0,1}$ the space of continuous function with limit 0 in $-\infty$ and $1$ in $+\infty$. Moreover, we have the convergence 
\begin{align}
		\zeta^{n} \longrightarrow \zeta, ~~\textrm{ where }~~
	\zeta(t):=\left(f'_{ISE}(t),f_{ISE}(t),\int_{-\infty}^t f_{ISE}(s)ds\right),
\end{align}
	where $f'_{ISE}$ is a continuous function, which is the derivative of $f_{ISE}$, (the convergence holds in distribution, for the topology of uniform convergence on compact sets). 
\end{theorem}
This theorem implies that $f_{ISE}$ has a continuous derivative as conjectured in~\cite[Conjecture 2.3]{BM-J}.
In fact, our proof of tightness uses a moment-increment estimate that proves that $f'_{ISE}$ is a.s. $\left(\frac{1}{2}-`e\right)$-H\"older for any $`e>0$. The value $\frac{1}{2}$ is sharp (Lemma \ref{lemma:momentBound} and Proposition \ref{prop:bonus}), which proves the other half of~\cite[Conjecture 2.3]{BM-J}. Note also that the theorem shows that $f'_{ISE}$ vanishes at the boundary of the support of $f_{ISE}$ (since it is continuous).

\subsection{A translated version of ISE on $\R^+$}
\label{sec:tv}

The fact that the CRT is invariant under uniform re-rooting (see \cite{Ald1990}) has various consequences. The re-rooting of a CRT with contour process $c=2\se$ at position $r\in[0,1]$, is defined thanks to its contour process
\[c^{(r)}(x) = c(x+r \mod 1)+c(r) - 2 \min \{c(u), u \in [ r \wedge (x+r \mod 1),  (r \vee (x+r \mod 1)]\},\]
(which gives the distance of the node encodes by $x+r$ to the node encoded by $r$ in the tree with contour process $c$).
The rerooted snake is obtained by rerooting the underlying tree at the new root, by setting 0 as its new spatial position, and by keeping the spatial variation along branches.
The head of the obtained brownian snake satisfies 
\[(H^{(r)}(x),x \in [0,1])=( H(x+r \mod 1)-H(r),x \in [0,1]),\]
and $H^{(r)}\eqd H$. Since for a uniform variable $u\in[0,1]$, independent from $H$, the distribution of $H(u)$ is given by the occupation measure of $H$ (which is $\mu_{ISE}$), one has:
\[\mu_{ISE} \eqd \mu_{ISE}( . - X)\]
where $X$ has distribution $\mu_{ISE}$ (to be clear: the translation value $X$ is taken under the random measure $\mu_{ISE}$, that it translates). These considerations allow one to understand that if ones take $L = \min {\sf Support}(\mu_{ISE})$, then \[\mu^+:=\mu_{ISE}( . - L)\] is a random measure on $[0,+\infty)$ which can be used to describe $\mu_{ISE}$:
\[\mu_{ISE}  \eqd \mu_{ISE}^+( . - Y)\]
where, again, $Y$ is taken under the random measure $\mu_{ISE}^+$. Let us call $\mu_{ISE}^+$ the translated version of $\mu_{ISE}$, and denote by $\zeta^{+}(t)=(f_{ISE}^{+}{}'(t),f_{ISE}^{+},\int_{-\infty}^{t}f_{ISE}^{+}(s)dt)$ the corresponding encoding processes.
The process $\zeta^+$ is, in nature, a bit simpler than $\zeta$ since "it starts" at a deterministic abscissa, when $\zeta$ has a bilateral random support.

\subsection{A companion process and a discrete diffusion}
\label{sec:comp}

Bousquet-Mélou and the first author~\cite{BM-C} have given a  complete description of the law of the vertical profile $M(\bT_n)$ under $\U_n$:
\begin{proposition}[{\cite[Thm. 1]{BM-C}}]\label{pro:comb-delta-M} 
Let $\ell \in \mathbb{Z}^-, r\in \mathbb{Z}^+$, $m_i \in \{1,2,\dots\}$ for any $i\in\cro{\ell,r}$ and $\sum_{i=\ell}^r m_i=n$, and $m_{\ell-1}=m_{r+1}=0$. We have 
\beq \label{eq:pro}
\!\!\!\!\#\{T \in {\cal B}_n:M(T)=(m_{\ell},\dots m_{-1}; m_0,\dots m_r)\}=\! \frac{m_0\binom{m_{-1}+m_1}{m_0-1}}{m_\ell m_r}\prod_{\ell \leq i \leq r\atop{i\neq 0}}\!\!\binom{m_{i-1}+m_{i+1}-1}{m_i-1},
\eq 
where $\binom{a}{b}=0$ if $b>a$.
\end{proposition}

Of course, for $\bT_n$ taken under $\U_n$,
$\mathbb{P}(M(\bT_n)=(m_{\ell},\dots m_{-1}; m_0,\dots m_r))$ is proportional to the right hand side (RHS) of \eref{eq:pro}. Rewrite this RHS a bit differently using that $\sum m_i=n$. For a normalising sequence $(\alpha_n,n\geq 1)$,  we have
\ben\label{eq:rep1}
\mathbb{P}(M(\bT_n)=(m_{\ell},\dots m_{-1}; m_0,\dots m_r))
&=&\frac{m_0\frac{m_1+m_{-1}}{m_{1}+m_{-1}-m_0+1}}{\alpha_n\,m_\ell m_r2^{m_{\ell}+m_r}} \prod_{\ell \leq i \leq r}\frac{\binom{m_{i-1}+m_{i+1}-1}{m_i-1}}{2^{m_{i-1}+m_{i+1}}}.
\een
\medskip

We now arrive at the main idea at the origin of this note:
the product in the formula of the law of $M(\bT_n)$ will lead us to observe that the process $M(T)$ can be (roughly) represented with the help of Markov chain $M^\star$.

However, because the $i$-th factor of the product depends on the numbers $m_{i+1}$ and $m_{i-1}$, to obtain a Markov chain representation we need to consider a tri-dimensional process: this is in fact the reason for the introduction of the process $Z(T)$.
Moreover, this Markovian representation will in fact hold only if we condition on the values of $Z(T)$ on the two boundaries of an interval.
We prove below that the companion process $M^{\star}$ (or rather $Z^\star$) possesses a diffusive limit: this will give the intuition that it should be also the case for $f_{ISE}$ (but this will not prove it, because of the difficulty of obtaining the same statement under a double conditioning).

\medskip

In order to parse \eref{eq:rep1}, recall that for any fixed positive $k$, the distribution defined by
\[p_k(n+k)=\binom{k-1+n}{k-1}2^{-n-k}, ~~~~n \geq 1,\]
is the binomial negative distribution $\BNEG(k)$; this distribution is that of the sum of $k$  i.i.d. geometric $1/2$ random variables $g^{(j)}$ (with support $\{1,2,\cdots,\}$, and then, mean 2):
\[p_k(n+k)=\P\l[\sum_{j=1}^k g^{(j)}=n+k\r]=\P\l[\sum_{j=1}^k (g^{(j)}-2)=n-k\r].\]
 Each factor in the product in the r.h.s. of \eref{eq:rep1} can thus be reinterpreted:
\ben  \label{eq:reecr22}
  \binom{(m_{i-1}+m_{i+1}-m_i)+m_i-1}{m_i-1}2^{-m_{i-1}-m_{i+1}}&=& \P\l(\sum_{j=1}^{m_i} (g^{(j)}-2)=\delta_{i+1}-\delta_i\r)
\een
where $\delta_{i}=m_i-m_{i-1}$.
Hence, if the prefactors in \eref{eq:rep1} were not there, then conditionally on $(M_k,k\leq i)$, the increment $\Delta_{i+1}=M_{i+1}-M_i$ would have the same distribution as $\Delta_i+\sum_{k=1}^{M_i} (g^{(k)}-2)$, and the process $(Z_j)$ would be a simple Markov chain (and $(S_j)$ would be a Markov Chain of order $3$).

 \medskip

To shed more light, we  introduce the \emph{companion process}, a time homogeneous Markov chain $(Z^\star_k,k\geq 0)= \l(\l(\Delta^\star_k,M^\star_k,S^\star_k\r),k\geq 0\r)$ taking its values in $\Z^3$ as follows:
conditionally on $(\Delta^\star_i,M^\star_i,S^\star_i)=(\delta_i,m_i,s_i)$, 
\beq\label{eq:DMS}
\left\{\begin{array}{cll}
\Delta^\star_{i+1} &=&\delta_i+\sum_{k=1}^{|m_i|} (g^{(k)}-2),\\
M^\star_{i+1}  &=& m_i+\Delta^\star_{i+1},\\
S^\star_{i+1} &=& s_i+M^\star_{i+1}. \end{array} \right.
\eq
Moreover, if $M^{*}_k \leq 0$, then the process $Z^\star$ is killed at time $k$ (meaning that $Z^\star_{k+t}=Z^\star_k$ for all $t\geq 0$).
Of course, to fully specify the process, we should specify a starting time and value -- and we will when needed. 

The companion process $Z^\star$ and the tree-related process $Z(\bT_n)$ are related as follows.
\begin{proposition}\label{prop:discreteDiff}
The distribution of the companion process $Z^\star$ coincides with $Z(\bT_n)$ on intervals which do not straddle 0, when one fixes boundary conditions at the two extremities of the interval: Formally, take $\bT_n$ under $\U_n$. Fix integers $0<k_1<k_2$ (or $k_1<k_2<0$)
	and $(\delta_1,m_1,s_1)$, $(\delta_2, m_2,s_2)$ in $\mathbb{Z}\times\mathbb{Z}_{>0}\times \mathbb{Z}$.
	Then the laws of the vectors $(Z^\star_j, k_1\leq j \leq k_2)$ and $(Z_j(\bT_n), k_1\leq j \leq k_2)$ conditioned to take the value $(\delta_i,m_i,s_i)$ at $k_i$ for $i=1,2$, are equal.
\end{proposition}
Observe that since $m_2>0$, this condition implies that under the conditional distribution, the second component of $Z^\star$ stays positive on $[k_1,k_2]$.
\begin{proof}
Take an element  $[z_j'=(d_j',m_j',s_j'), k_1\leq j \leq k_2]$ such that for all $k_1<j<k_2$, $m_j'>0$, $m_{j}'=m_{j-1}'+d_j', s_j'=s_{j-1}'+m_j'$ (so that $(s_j')$ is increasing), and moreover at the boundary $z_{k_j}'=(\delta_j,m_j,s_j)$ for $j \in \{1,2\}$. Using \eref{eq:reecr22}, it is immediate to check that $\P(Z^\star_j=z_j', k_1\leq j \leq k_2 ~|~Z^\star_{k_1}=z_{k_1}', Z^\star_{k_2}=z_{k_2}')$ and $\P(Z_j(\bT_n)=z_j', k_1\leq j \leq k_2 ~|~Z_{k_1}(\bT_n)=z_{k_1}', Z_{k_2}(\bT_n)=z_{k_2}')$ are proportional, and then, are equal.
\end{proof}

Notice that the prefactor in \eref{eq:rep1} says something about the root position (the position of "0" in the interval $[\ell,r]$), as well as a kind of cost of the extremal values $(M_{\ell},M_r)$. For simplicity, we have chosen in Proposition \ref{prop:discreteDiff} to consider only intervals avoiding zero. This allows one to work with nicer formulas. Section \ref{sec:tv}, in which the translated version $\zeta^{+}$ is introduced, suggests that the root position is not important, and can be thought to be close to the left support.

Note also that the re-rooting invariance of the continuum model (Section \ref{sec:tv} again) is not exactly present in discrete binary trees. 

Finally, note that the prefactor in \eref{eq:rep1} and the global condition of positivity of $M(\bT_n)$ on $\cro{\ell,r}$ make the global study of \eqref{eq:rep1} quite difficult; with intervals avoiding zero we avoid these difficulties.

\subsection{Convergence to a diffusion for the companion process attached at the left boundary.}
\label{sec:diffApprox}

In view of \eqref{eq:DMS}, and since the sum $\sum_{k=1}^{|m_i|} (g^{(k)}-2)$ of centred i.i.d. variables should be well approximated by a Gaussian of variance $|2m_i|$, it can be expected that $Z^\star$ (started at time 0) will converge in distribution, after an appropriate rescaling, to a process 
$\zeta_t^* = (\delta^*_t,m^*_t,s^*_t)$
solution to the stochastic differential equation
\ben\label{eq:hdgfygu} 
\delta_t^\star &=& \delta_0+\int_{0}^t \sqrt{2m_x^\star}\, dW(x),~~
m_t^\star=m_0+\int_0^t \delta^\star_x\, dx,~~ 
s_t^\star=s_0+\int_0^t m_x^\star dx
\een
where $W(t)$ is a standard Brownian motion.
Note that the dynamics of the process can be encapsulated in the unique SDE
\begin{align}\label{eq:SDEcompact}
d \big( (m^*_t)' \big) = \sqrt{2m^*_t} dW_t,
\end{align}
which is some "order-2" diffusion (note that the S-coordinate plays no direct role in the dynamics).

 Since $x\mapsto \sqrt{x}$ is not Lipschitz (at 0) and $M^\star$ is stucked when it becomes negative, however, some precautions will be needed.

\paragraph{A stuck version of the companion process.}

The process $Z^\star$ is well defined for any initial distribution with support in $\mathbb{Z}^3$, however, we are only interested in its behaviour when $M^\star\geq 0$ (we will condition on that event). Moreover, the convergence result we are about to state needs the diffusion coefficient to be positive: we choose to stuck $Z^\star$ when a certain level $K>0$ is hit by  $M^\star$. 
For $K>0$ fixed, denote by
\[T_K= \inf\{t >0: M^\star_t<K\}\]
we define $\P_{\geq K}$ the distribution of the Markov chain $\bar{Z}^{[K]}$ obtained from $Z^\star$ as follows:
\[\bar{Z}_t^{[K]}= Z^\star_{t \wedge T_K}.\]
The process $Z^\star$ is time homogeneous. To describe its limit, we work on the time space $\{0,1,2,\cdots\}$.
Let us fix a constant $T>0$, and set, for all $k\geq 0$, $n \geq 1$,
\[\tkn{k} = kn^{-1/4}, ~~N_n=\min\{k:\tkn{k}>T\}=\floor{Tn^{1/4}}+1.\]
Set $\xi^{n,\star}=(\delta^{n,\star},m^{n,\star},s^{n,\star})$ as the càdlàg process, constant on $[\tkn{k},t_{k+1}^n)$, and satisfying
\[\xikn{}(\tkn{k}) =\bma n^{-1/2} \Delta_{k}^\star \\ n^{-3/4} M_{k}^\star\\n^{-1} S_k ^\star\ema .\] 
The process $\l(\xikn{}(\tkn{k}), k\geq 0\r)$ is a Markov chain.

As usual, denote by  $D([0,T],\R^3)$ the set of càdlàg functions defined on $[0,T]$ taking their values in $\R^3$, equipped  with the Skorokhod topology. We write $\bar{\xi}^n$ for the stucked version.  
\begin{theorem}\label{thm:diffApprox} Let $(\delta_0,m_0,s_0) \in \R \times (0,+\infty) \times \R$, $T>0$ fixed, and $(\delta^n_0,m^n_0,s^n_0)$ be a sequence in $\Z^3$ such that $m^n_0>0$.
If 
\[z_0^n:=\l(\delta^n_0/n^{1/2},m^n_0/n^{3/4},s^n_0/n\r)\to z_0:=(\delta_0,m_0,s_0)\] then, for any $0<`e<m_0$, there is a Brownian motion $(W(t),t\geq 0)$ and a random process $\xi^\star(.)$ non-anticipative with respect to $W(.)$ so that,
  under $\P_{\geq `e n^{3/4}}$, $(\bar{\xi}^n)$   starting at initial position $z_0^n$,  converges in distribution in $D([0,T],\R^3)$  to $\xi^\star$, unique solution of the stochastic differential equation
  \ben\label{eq:qfge}
  \xi^\star(t)=\xi(0) +\int_0^t f(\xi^\star(s),s) ds + \int_0^t \sigma(\xi^\star(s),s) dW(s)
  \een
  where $\xi^\star(0)= z_0$, stuck when its second entry hits $`e$ and
  \[f\l(\bma d\\m\\s\ema,t\r)=\bma 0\\ d \\ m\ema,~~\sigma\l(\bma d\\m\\s\ema,t\r)=\bma \sqrt{2|m|}\\0\\0\ema. \] 
  This means that, before being stuck when $m_t^\star=`e$,  $\zeta^\star(t)=(\delta^\star_t,m^\star_t,s^\star_t)$ satisfies
  \eref{eq:hdgfygu}.
  \end{theorem} 

Unfortunately, we are not able to prove a ``local-limit'' statement that would be the analogue of Theorem~\ref{thm:diffApprox} when the processes are conditioned by their value on the two boundaries of an interval of the form $[t_1,t_2]$. 
Moreover, since the process $f_{ISE}$ vanishes at the boundary of its support (since it is a.s. continuous on $\R$), the approximation given by Theorem~\ref{thm:diffApprox} is not sufficient to describe entirely the process $\zeta^+$ (since $`e>0$).

If we could overcome these difficulties, we would obtain by Proposition~\ref{prop:discreteDiff} that the process
$\zeta$ (defined in Theorem \ref{thm:tightness})
behaves as a diffusion, at least on any compact sub-interval of its support.

We hope that experts of diffusion approximations could be able to bridge these gaps. In the next section, we ask explicit questions in this direction.

\subsection{Some questions and conjectures}
\label{seq:sq}

\subsubsection{Questions on ISE}

 Maybe the most direct question that follows from the previous discussion is to know whether one can add a $dt$ term to~\eqref{eq:SDEcompact} to obtain an SDE that would completely describe the process $\zeta$. It is more natural to formulate it for the translated process $\zeta^+$ (defined in Section~\ref{sec:tv}) to avoid dealing with the bias at zero.
\begin{conjecture}\label{conj:main}
There is a continuous function $g$ such that 
the following stochastic differential equation holds, for $t>0$:
$$
d \big( f_{ISE}^{+}{}'(t) \big) = \sqrt{2f_{ISE}^+(t)} dB_t + g\left(f_{ISE}^{+}{}'(t), f_{ISE}^+(t),\int^t_{-\infty} f_{ISE}^+(s)ds \right)dt.
$$
\end{conjecture}
The conjecture implies a similar equation for the unshifted process $\zeta$ on $(0,+\infty)$. The law on $(-\infty,0)$ should be more complex since each trajectory has to be biased by the value $f_{ISE}(0)$ (see the rerooting property in Section~\ref{sec:tv}). It is natural to expect a similar SDE in which the function $g$ depends on a fourth parameter, $t$.

Proving that conjecture would be very interesting, especially if the function $g$ can be expressed explicitly.
A possible approach to this question would be to try to re-sum the product formulas of~\cite{BM-C} to obtain, at the discrete level, the explicit multivariate generating functions encoding the conditional transition probabilities for the process $\zeta^n$. While we believe that may be approachable while staying in the realm of algebraic functions (recall that a function is algebraic if it is the solution of a non-null polynomial equation whose variables are the function itself and its parameters), the subsequent analytic combinatorics in several variables required might lead to important technical difficulties. We hope that a direct approach from the continuum could lead to better solutions.
In any case, this suggests that the function $g$ in Conjecture~\ref{conj:main} could be algebraic and even quite explicit.

Another natural goal would be to try to prove the missing ``local limit'' version of Theorem~\ref{thm:diffApprox}. 
\begin{conjecture}[two-sided version of Theorem~\ref{thm:diffApprox}]\label{conj:diffApprox2}
For $i=0,1$, let $(\delta_i,m_i,s_i) \in \R \times (0,+\infty) \times [0,1]$, and $(\delta^n_i,m^n_i,s^n_i)$ be a sequence in $\Z^3$ such that $m^n_i>0$.
Assume 
\[z_i^n:=\l(\delta^n_i/n^{1/2},m^n_i/n^{3/4},s^n_i/n\r)\to z_i:=(\delta_i,m_i,s_i), \textrm{ for }i\in\{0,1\}.\]
Let $0<`e<s_1$, and consider the process $\xi^\star(t)$ as in Theorem~\ref{thm:diffApprox}, stuck when $m^\star$ hits $`e$, and started at $z_0$ for $t=0$.
Then, the discrete process $\P_{\geq `e n^{3/4}}$, $(\bar{\xi}^n)$ started at time $0$ at position $z_0^n$ and conditioned to take value $z_1^n$ at time $t_1$, converges in distribution in $D([0,t_1],\R^3)$  to $\xi^\star$, started at position $(\delta_0,m_0,s_0)$ conditioned by $\xi^{\star}(t_1)=z_1$.
\end{conjecture}
In the last sentence, notice that we can condition the process $\xi^{\star}$ by its terminal value $\xi^{\star}_{t_1}$ on $[0,t_1]$, on the support of this variable -- and this characterizes (by disintegration), up to a negligible set, the conditional distribution given the ending point. It can be shown (personal communication of Nicolas Fournier \cite{NF}), that $m^\star_{t_1}$ is absolutely continuous with respect to the Lebesgue measure on any compact subset of  $(0,+\infty)$.

The last conjecture seems to be related to ''continuity'' properties of the law of the process $\zeta^{n,*}$ on an interval $[k_1,k_2]$ and conditioned to its right boundary, according to the value on that right boundary. Although it seems difficult to obtain such a result in the general framework of approximating discrete Markov processes by diffusions, it might be doable for this particular case.

From the viewpoint of the convergence of $\zeta^{n}$, and given Proposition~\ref{prop:discreteDiff}, it would we even better to establish the following:
\begin{conjecture}\label{conj:diffApprox3}
Conjecture~\ref{conj:diffApprox2} also holds with $`e=0$.
\end{conjecture}

\medskip
 Note that Formula~\eref{eq:pro} is reminiscent of the closed formula for the distribution of ``horizontal profile'' that one may find for rooted plane trees, or rooted Cayley trees, taken under the uniform distribution on the corresponding sets of trees with $n$ vertices. The horizontal profile, in general just called ``profile'' in the literature, is the sequence $(z_i,i\geq 0)$ of the number of vertices at successive levels in the tree. If $h$ is a positive integer, the number of rooted plane tree having $z_i>0$ vertices at level $i$, for $1\leq i \leq h$ is given by $\prod_{i=0}^{h-1} \binom{z_i+z_{i+1}-1}{z_{i}-1}$, where $z_0=1$: indeed, $(c_1,\dots, c_{z_i})$ the number of children of the $z_i$ individuals at level $i$ forms a composition of $z_{i+1}$ (this is well known, see e.g.~\cite{BM-C} again).
The horizontal profile in these trees converges after space and time normalization to the local time of the Brownian excursion, which is the solution of a stochastic differential equation, as proven by Pitman \cite{jp97sde} (see also Drmota \& Gittenberger \cite{D-G}). As a starting point to approach the questions above, one might try to reprove this differential equation directly from the discrete product formula, via a diffusion approximation and conditioning of the natural companion Markov process.

\subsubsection{Family of Distributions subject to Boundary Conditions (FDBC)}

Finally, this discussion raises several questions about the characterization of a random continuum process by its law on proper compact subintervals of the support. Assume for the sake of the discussion, that a one dimensional process $X$ has a continuous version on  $[0,1]$. Consider the \emph{family of distributions subject to boundary conditions} (FDBC) of $X$, which is the data of all the laws
$${\cal L}\big( [X(t), t\in[t_1,t_2]]~|~X(t_1)=x_1,X(t_2)=x_2\big)$$ for all $0<t_1<t_2 <1$ (observe that 0 and 1 are excluded) and all $(x_1,x_2)$ in the support of $(X(t_1),X(t_2))$. We observe that, in general, the FDBC is \emph{not} sufficient to characterize the distribution of~$X$. For example, the Brownian motion and the standard Brownian bridge have the same FDBC. Even the knowledge of the FDBC and the knowledge of the boundary distribution (for example $X(0)=X(1)=0$) is not sufficient to characterize $X$: for example, if 
$$X^{(p)}= -{\sf Ber}(p)\se+(1- {\sf Ber}(p))\se$$ where $\se$ is a normalised Brownian excursion, and ${\sf Ber}(p)$ an independant Bernoulli random variable with paremeter $p$, then  $X^{(p)}_0=X^{(p)}_1=0$ for all $p$, the $X^{(p)}$ have same FDBC for all $p\in(0,1)$, but the law of $X^{(p)}$ clearly depends on $p$.

What can be shown to be sufficient to determine the distribution of $X$ is, in addition to $X(0)=X(1)=0$ and of the knowledge of the FDBC, the continuity\footnote{ A notion of continuity which is sufficient here, is the convergence of finite dimensional distributions under ${\cal L}\big( [X(t), t\in[t_1,t_2]]~|~X(t_1)=x_1,X(t_2)=x_2\big)$ when $(t_1,t_2,x_1,x_2)\to (0,1,0,0)$.} of the map $(t_1,t_2,x_1,x_2)\mapsto {\cal L}([X(t), t\in[t_1,t_2]]~|~(X(t_1)=x_1,X(t_2)=x_2)$ in $(0,1,0,0)$.

The weakness of the FDBC to characterize the distribution of processes implies that the convergence of the FDBC is also a much too coarse tool to entail convergence in distribution. That being said, of course, such a convergence still carries an important amount of information and proving it in the case studied in this paper (Conjecture~\ref{conj:diffApprox2} or~\ref{conj:diffApprox3}) would be very interesting.

\section{Proof of tightness}
\label{sec:proofTight}

In this section we prove Theorem~\ref{thm:tightness}.
First note that by the results in~\cite{BM-J} (see Prop. \ref{prop:BMJ}), the tightness of the sequence of processes
$ \left( n^{-3/4} M_{n^{1/4} t}^{n} 
 \right)_{t\in \mathbb{R}}  
$ in $C_K(\R,\R)$
and its convergence to $f_{ISE}$ are known. So to prove tightness of $\zeta^n$, it suffices to prove the tightness of its first component, namely the sequence of processes
$ \left( n^{-1/2} \Delta_{n^{1/4}t}^{n}   \right)_{t\in \mathbb{R}} $.
Moreover, assuming this is done, pick any sub-sequential limit $g$ of that process (in distribution) and a probability space on which this convergence is realized almost surely, then it follows from the discrete identity
$
M^n_k = \sum_{j\leq k} \Delta^n_j
$
that, almost surely, we have $\int_{-\infty}^t g(s)ds = f_{ISE}(t)$. Therefore $g=(f_{ISE})'$  which shows that the convergence holds without taking subsequences -- and that the limit is, as claimed, the derivative of $f_{ISE}$.

Therefore we only have to prove the tightness of $ \left( n^{-1/2} \Delta_{n^{1/4}.}^{n}   \right) $.
For this, we will rest on the well known moment criterion of~Billingsley~\cite[Theo. 12.3]{BIL68}\footnote{Note that we reference to first edition of Billingsley's book here.} , whose main point is to prove the following lemma. In fact, any even value of $p\geq 4$ is sufficient for the moment criterion, but the proof is the same for arbitrary $p$ so we state it in this form. 
\begin{lemma}\label{lemma:momentBound}
	Let $p\geq 1$.
	Let $a\leq b \in \mathbb{Z}$ such that $\left|\frac{b-a}{n^{1/4}} \right|\leq 1$.
Then one has
	\begin{align}\label{eq:mainMomentBound}
	\left|	\mathbf{E}\left[\left(\frac{\Delta^n_b-\Delta^n_a}{\sqrt{n}}\right)^{p}\right] 
	\right|
	\leq C_p \left(\frac{b-a}{n^{1/4}}\right)^{ \lceil p/2 \rceil},
\end{align}
for some constant $C_p$ depending only on $p$.
\end{lemma}
The constant $1$ in the modulus upper-bound plays no role in the proof of the lemma nor in its application (any positive constant would work).

	\subsection{Proof of Lemma~\ref{lemma:momentBound}}
	Let us first sketch the method, which is relatively straightforward even if it takes space to write: we will expand the $p$-th power, interpret the quantities obtained as counting (with signs) trees with $p$ marked vertices, and proceed with generating functions. We will compute these generating functions combining a skeleton decomposition (whose origin in the context of labelled trees might be traced back to~\cite{CMS}) with analytic combinatorics via Hankel contours. The main point of the proof is that, after the proper analysis is done and sign cancellations analysed carefully, we can identify the dominating configurations as the ones in which the $p$ marked vertices are grouped together ''in $\lfloor \frac{p}{2}\rfloor$ pairs'',
	which \emph{in fine} is the explanation for the ratio between exponents $p$ and $\frac{p}{2}$ in \eqref{eq:mainMomentBound}. This appearance of pairs in the analysis is in some sense the combinatorial incarnation of the Gaussian nature of the underlying continuum process, see also Section~\ref{sec:gaussianPairs}. Let us now proceed with the proof. 

\medskip

	Let $q=b-a$ and assume $q\geq 1$. Write $q=\mu n^{1/4}$. Since the Catalan numbers satisfy $Cat(n) = \Theta(4^n n^{-3/2})$, we need to show that 
	\begin{align}\label{eq:toProve}
	|	Cat(n)\times \mathbf{E}\left[\left(\Delta^n_{b+1}-\Delta^n_{a+1}\right)^{p}\right] |\leq C_p 4^n n^{\frac{p}{2}-\frac{3}{2}} \mu^{\lceil\frac{p}{2}\rceil},
	\end{align}
	for some constant\footnote{in the rest of the proof, the notation $C_p$ will denote a constant, that may vary from line to line, but that will always be dependent only on $p$ (and not on $b,a,q,\mu,n$).} $C_p$. Note that compared to~\eqref{eq:mainMomentBound} we have incremented $a$ and $b$ by one, which  will be convenient for notation.

	\subsubsection{Marked trees, skeletons}
	Since $\Delta^n_{b+1}-\Delta^n_{a+1}= M^n_{b+1}-M^n_b-M^n_{a+1}+M^n_a$, we can write
	\begin{align*}
		\mathbf{E}\left[\left(\Delta^n_{b+1}-\Delta^n_{a+1}\right)^{p}\right]
		= \sum_{`e_1,\dots,`e_{p} \in \{0,1\}\atop `e'_1, \dots, `e'_{p} \in \{0,1\}} (-1)^{\sum `e_i+\sum`e'_i} 
		E\left[\prod_{i=1}^{p} M^n_{a+q`e_i+`e'_i}\right]
\end{align*}
so 
	\begin{align}\label{eq:bigsum}
		Cat(n) \times 
		\mathbf{E}\left[\left(\Delta^n_{b+1}-\Delta^n_{a+1}\right)^{p}\right]
		= \sum_{`e_1,\dots,`e_{p} \in \{0,1\}\atop `e'_1, \dots, `e'_{p} \in \{0,1\}} (-1)^{\sum `e_i+\sum`e'_i} 
	T_n(a+q`e_i+`e'_i,i=1\dots p),
\end{align}
	where $T_n(i_1,\dots,i_{p})$ denotes the number of binary trees of size $n$ having $p$ (numbered, possibly repeated) distinguished vertices, of respective abscissa $i_1,\dots,i_{p}$.

	We will evaluate the sum~\eqref{eq:bigsum} by grouping trees with $p$ marked vertices according to their skeleton and their scheme, which we now define.	
\begin{figure}
	\begin{center}
		\includegraphics[height=50mm]{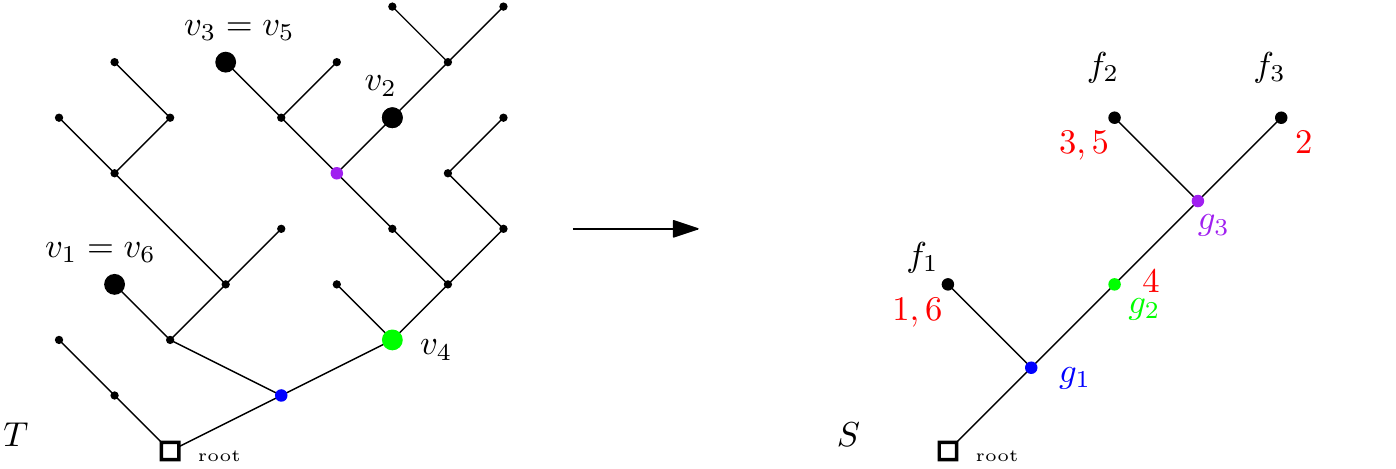}
		\caption{Left: a binary tree $T$ with six marked vertices. Right: Its skeleton $S$. To help vizualize which vertices of $T$ are present in $S$, we represented them with colors. The decorations $1,2,\dots,6$ indicating the position of the vertices $v_i$  are in red.}
		\label{fig:skeleton}
	\end{center}
\end{figure}
\begin{definition}[skeleton, scheme, Figure~\ref{fig:skeleton}]\label{def:skeleton}
		Let $T$ be a rooted binary tree with root $\rho$, with $p$ marked vertices $v_1,\dots,v_{p}$, numbered and possibly repeated. 
		 
		Let $V$ be the set formed by the vertices $v_1\dots, v_{p}, \rho$ together with all their pairwise highest common ancestors.
		The \emph{skeleton} $S$ of $T$ is the rooted binary tree on the vertex set $V$ obtained from $T$ by iteratively removing all vertices which are leaves but are not in $V$, until no such leaf remains, and then replacing each path of vertices of degree $2$ joining two points of $V$ (but containing no other vertex of $V$) by a single edge. 
				We preserve the left/right order of edges outgoing from vertices of $V$, so $S$ has a structure of binary tree.  
We write (without repetition)
	$$V = \{f_1,\dots,f_k,g_1,\dots,g_\ell, \rho\},$$ 
		where $f_1,\dots,f_k$ are the leaves (vertices without children) of $S$, and $g_1,\dots g_\ell$ are either unary or binary vertices ($1$ or $2$ children).
			We number the $f_i$ and the $g_i$ in the natural depth-first order of the tree. Some of these vertices are also naturally decorated by one or several integers from $1$ to $p$, which record the position of the vertices $v_i$ in  $T$ (this decoration is part of data of the skeleton, see Figure~\ref{fig:skeleton} again).
				Note that a vertex $f_i$  necessarily carries a decoration and the same is true for a unary vertex $g_i$, but a binary vertex $g_i$ can be decorated or not.
	
	The \emph{labelled skeleton} of the marked tree $T$ is the pair $\hat S=(S,x)$, where $S$ is the skeleton and  $x=(x_1,\cdots,x_\ell)$ is the sequence of abscissas $x_1,\dots,x_\ell$ of $g_1,\dots,g_\ell$ in $T$.

		The \emph{scheme} of $T$ is the pair $\tilde S=(S,\lambda)$ where $S$ is the skeleton of $T$ and $\lambda=(\lambda(1), \dots, \lambda(\ell))\in\{0,1,2\}^\ell$, where for each $i\in [\ell]$, we have $x_i\in I_{\lambda(i)}$ where $I_0 \cup I_1 \cup I_2$ is the following partition of $\mathbb{Z}$:
\begin{align}\label{eq:split}
	\mathbb{Z} =
	\underbrace{(\infty,a-1] \cup [b+2,\infty)}_{I_0}
	\cup
	\underbrace{[a+2,b-1]}_{I_1}
	\cup
	\underbrace{\{a,a+1,b,b+1\}}_{I_2}.
\end{align}

	The number of non-root inner vertices whose abcissa belongs to each of these three sets will play a key role in what follows. We let  
		$$
		n_0 = |\lambda^{-1}(\{0\})| \ \ , \  \  
		n_1 = |\lambda^{-1}(\{1\})| \ \ , \ \  
		n_2 = |\lambda^{-1}(\{2\})|. 
		$$
	
		Note that the scheme $\tilde S$ can be infered from the labelled skeleton $\hat S$. We say that $\hat S$ is \emph{compatible with $\tilde S$} if trees of labelled scheme $\hat S$ have scheme $\tilde S$.
	\end{definition}

	\begin{remark}\label{rem:relabelings}
		There are finitely many binary trees having at most $p$ leaves and unary vertices. Moreover, there is a finite number of ways to distribute the decorations $1,\dots,p$ of $v_1,\dots,v_{p}$ among the vertices of $S$. So there is a finite number of skeletons -- and a finite number of schemes.
\end{remark}

	In the sum~\eqref{eq:bigsum}, we will group the configurations having the same labelled skeleton, and we will evaluate (upper bound) their contribution separately.

\subsubsection{Branches and their generating function}
\label{subsubsec:branches}

\begin{figure}
	\begin{center}
		\includegraphics[height=3.0cm]{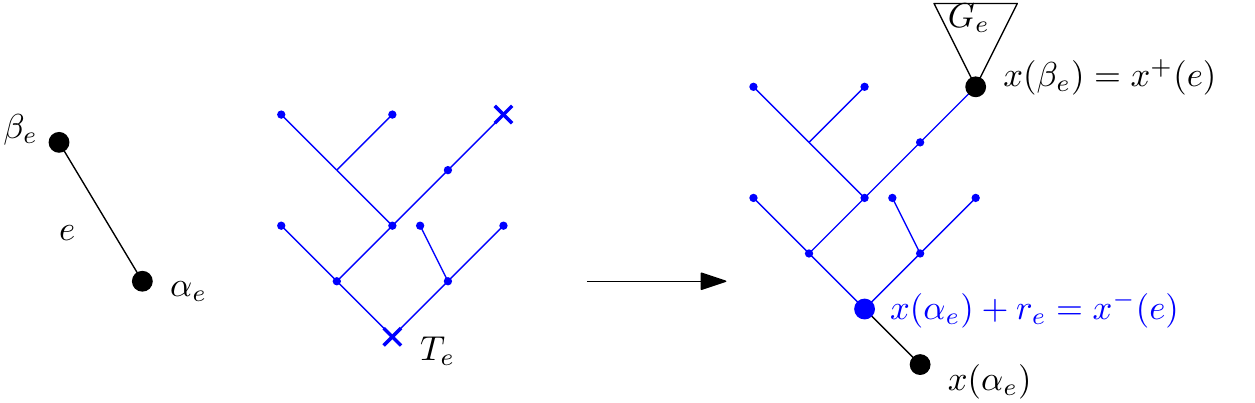}
		\caption{Replacing an edge $e$ of the labelled skeleton by a branch $T_e$. On this example $e$ is an left-leaning edge ($r_e=-1$) and the increment of the branch to be substituted is $x(\beta_e)-(x(\alpha_e)-1)=2$.}
		\label{fig:branch}
	\end{center}
\end{figure}

Fix a labelled skeleton $\hat S=(S,x)$, and use the notation of Definition~\ref{def:skeleton}.
We will explain how to construct all trees with labelled skeleton $\hat S$ by substituting edges with branches. A \emph{branch} is just a rooted binary tree with a marked leaf, considered up to additive translation of all abscissas. The \emph{increment} of the branch is the abscissa of its marked leaf (when the root is set to zero). The abscissas of vertices appearing along the path from the root to the mark leaf form a walk on $\mathbb{Z}$ with steps $\pm 1$.

Now, for each edge $e$ of $S$, let $\alpha_e \in V$ be the mother  vertex of $e$ in $S$, and let $\beta_e$ be the other endpoint of $e$.
Let $x(\alpha_e), x(\beta_e)$ be their abscissa in the labelled skeleton (for the moment, the second one is defined only if $\beta_e$ is not a leaf, since the labelled skeleton, by definition, does not record the abscissas of its leaves). 
Let $r_e$ be equal to $+1,-1$, respectively if $e$ is leaning right or left  from $\alpha_e$.

All the marked trees contributing to~\eqref{eq:bigsum} have $p$ marked vertices (possibly repeated) whose abscissa is among $\{a,a+1,b,b+1\}$. Among these trees, 
the trees $T$ having labelled skeleton $\hat S$ can be obtained in a unique way by the following procedure
(Figure~\ref{fig:branch}).
\begin{itemize}[itemsep=0pt, topsep=0pt,parsep=0pt, leftmargin=12pt]
	\item for each leaf $f$ of $S$, choose its abscissa $x(f)$ among $a,a+1,b,b+1$,
	\item  for each edge $e$ of $S$, let $x^-_{e}:=x(\alpha_e)+r_e$ and $x^+_e:=x(\beta_e)$. Replace the edge $e$ by a branch $T_e$ of increment $x^+(e)-x^-(e)$, whose root and marked leaf are respectively linked to $\alpha_e$ and identified with $\beta_e$,
	\item attach a rooted binary tree $G_e$ to each $\beta_e$ which is a leaf of $S$.
\end{itemize}

In order to evaluate the expression~\eqref{eq:bigsum}, we consider its generating function, in a new variable~$t$,
\begin{align}\label{eq:gf}
\sum_{n\geq 1}  \sum_{`e_1,\dots,`e_{p} \in \{0,1\}\atop `e'_1, \dots, `e'_{p} \in \{0,1\}} (-1)^{\sum `e_i+\sum`e'_i} 
	T_n(a+q`e_i+`e'_i,i=1\dots p) t^n.
\end{align}
Note that the coefficient of $t^n$ in~\eqref{eq:gf} is precisely~\eqref{eq:bigsum}.
We let $F_{\hat S}(t)$  the contribution of the set of marked trees $T$ having a given labelled scheme $\hat S$ to the generating function~\eqref{eq:gf}. This generating function can be evaluated by computing separately the contribution of each edge of $\hat S$, as we now show. In all generating functions below, the variable $t$ will mark the number of vertices of the underlying trees.

To start with, using classical last-passage decompositions, the generating function of branches of increment $i$ is found\footnote{Proof: the series $T(t)$ is clearly the generating function of binary trees by the number of vertices. By the bijection between trees and Dyck path, the series $U(y)$ is the generating function of Dyck paths extended by an extra up-step, and the series $B(y)=\sum_{n\geq0}{2n \choose n} y^{2n}$ is the generating function of Dyck bridges ($\pm1$ path going from $0$ to $0$). By decomposing a walk ending at $i\geq 0$ at the last passage at each integer $j\in [0,i]$, we decompose it into a bridge followed by $i$ translated Dyck paths separated by single up steps, which shows that the series $BU^i$ is the generating function of $\pm1$ walks ending at $i$ (note that the walk can be empty, which is not a problem for us).

This walk can be interpreted as a left/right path going from the root to a marked leaf of abscissa $i$.
Finally, the relation $y=t(T-1)$ means that a (non-empty) binary tree is to be substituted at each step of that path.} to be
\begin{align}\label{eq:Hi}
H_i(t) =  B U^{|i|},
\end{align}
where the series $U\equiv U(t)$ is defined by the system of equations
\begin{align}\label{eq:gfs}
\left\{\begin{array}{c} T=t(1+T)^2, \\ U=y(1+U)^2, \\ y=t(T-1),\end{array}\right.
\end{align}
and $B = (1-4y)^{-1/2}=\frac{1+U}{1-U}$.
The critical point of this system is readily found to be equal to $t=t_c:=\frac{1}{4}$, corresponding to $U=U_c := 1$. More precisely, the algebraic series $B(t)$ and $U(t)$ have a unique singularity of minimal modulus at $t_c$, close to which they admit the  Puiseux expansions
\begin{align}\label{eq:Puiseux}
	U = 1 - c (1-4t)^{1/4} + O((1-4t)^{1/2}) \ \ , \ \ 
	B = c' (1-4t)^{-1/4} + O(1) , 
\end{align}
with constants $c, c'>0$ that we do not need to make explicit.
In particular, there exists a  neighbourood $V$ of $\frac{1}{4}$ in $\mathbb{C}\setminus(\frac{1}{4},\infty)$ on which $|U(t)|$ has a unique maximum (equal to $1$) at $t=\frac{1}{4}$.

We can now  compute the generating function $F_{\hat S}(t)$. We have  
\begin{align}\label{eq:FShat}
	F_{\hat S}(t) = `e_{\hat S} \cdot t^\ell T^k \cdot 
	\prod_{e\in E(S)} W_e(t),
\end{align}
where $E(S)$ is the set of edges of the skeleton, and $W_e(t)$ is the generating function of the branches that can be substituted at edge $e$, discussed below. 
Here the factors $T^k$ and $t^\ell$ account respectively for the trees $G_e$ to be inserted at each leaf $\beta_e$, and for the internal vertices of $S$. The sign $`e_{\hat S}\in \{\pm 1\}$ is the contribution to the sign in~\eqref{eq:bigsum} of the decorated vertices which are one of the inner vertices $g_1,\dots,g_\ell$ of the skeleton\footnote{Namely, for each $i$ from $1$ to $p$, if the $i$-th decoration in $S$ is carried by the vertex $g_j$ of label $x_j$, then necessarily $x_j \in \{a,a+1,b,b+1\}$ and the multiplicative contribution of this index $j$ to $`e_{\hat S}$ is $+1$ if $x_j\in \{a,b+1\}$ and $-1$ otherwise. This sign will play no role in what follows as we will only estimate the modulus of $F_{\hat S}(t)$.}.

The function $W_e$ depends on the case considered:
\begin{itemize}
	\item ''internal'' edge: $\beta_e$ is not a leaf of $S$. Then the only constraint on the branch is its increment, which has to  equal $x^+(e)-x^-(e)$, therefore 
		\begin{align}\label{eq:WeInterne}
			W_e = H_{x^+_e-x^-_e} = BU^{|x^+_e-x^-_e|}.
		\end{align}

	\item ''external'' edge: $\beta_e$ is a leaf of $S$. We distinguish two cases:
	 \begin{itemize}
	 \item $\beta_e$ carries a unique decoration. If this is the case, in trees of scheme $S$, $\beta_e$ corresponds to a unique marked vertex $v_i$, and it has to be counted, in~\eqref{eq:bigsum}, with a sign that depends on the value of $x(\beta_e)$ in $\{a,a+1,b,b+1\}$.
		
		Summing over these four possible values, we get the contribution
		\begin{align}\label{eq:Wesum}
			W_e = \sum_{`e , `e' \in \{0,1\}} (-1)^{`e+`e'} 
			B U^ {|a+q`e+`e'-x^-_e|}.
		\end{align}
	
 To evaluate this expression we need to get rid of absolute values, which requires to know the relative position of $x:=x^-_e$ with numbers $a, a+1, b, b+1$.
		Note that  this information is present in the labelled skeleton. 
		The sum~\eqref{eq:Wesum} is immediately  computed in each case:
$$
\begin{array}{ll}
	case & W_e  
	\\  x\leq a & 
			B(1-U)U^{a-x}(1-U^q) 

	\\  x=a+1 & 
			B(U-1)(1+U^{q-1})  
	\\  x\in(a+1,b-1) &
			B(U-1)(U^{b-x}+U^{x-a-1})  
	\\  x=b & 
			B(U-1)(1+U^{q-1})  
	\\ x\geq b+1 & 
			B(1-U)U^{x-(b+1)}(1-U^{q}).
\end{array}
$$
		\vspace{-3cm}
		\begin{align}\label{eq:tableWeExterne}
		\end{align}
		\vspace{1cm}
		
 \item $\beta_e$ carries $h$ decorations with $h\geq 2$. We call the edge $e$ \emph{frozen}. If this is the case, in trees of scheme $S$, $\beta_e$ corresponds to $h$ coinciding marked vertices $v_i$, all contributing to the same sign. We thus get 
 	\begin{align}\label{eq:WesumF1}
			W_e = \sum_{`e , `e' \in \{0,1\}} \big((-1)^{`e+`e'} \big)^h
			B U^ {|a+q`e+`e'-x^-_e|}.
		\end{align}
  In this case we will only need to use the modulus upper bound
  \begin{align}\label{eq:WesumF2}
			|W_e| \leq \sum_{x \in \{a,a+1,b,b+1\}}|B| |U|^ {|x-x^-_e|}.
		\end{align}
		  Note that we could be more precise when $h$ is odd, but we will not need this.
   \end{itemize}

\end{itemize}

\subsubsection{Hankel contours}

\begin{figure}
	\begin{center}
		\includegraphics[height=3.5cm]{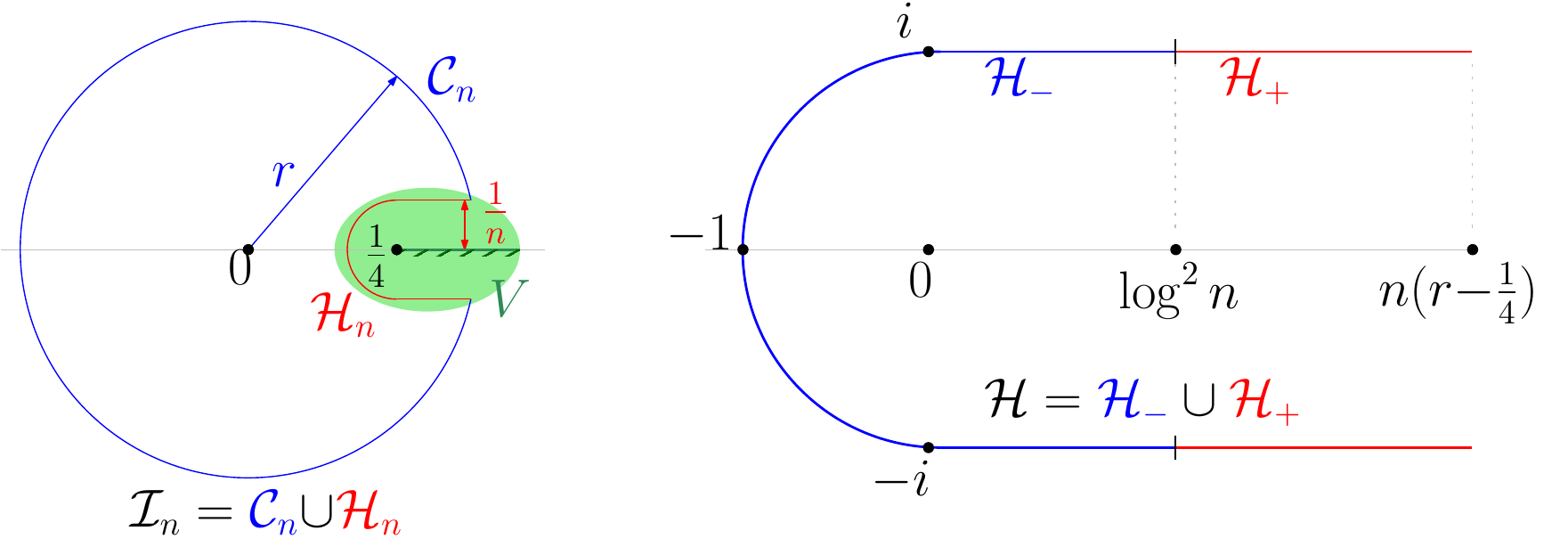}
		\caption{Contours of integration. Left: $t$-plane; Right: $\tau$-plane.}
		\label{fig:hankel}
	\end{center}
\end{figure}

For a scheme $\tilde S=(S,\lambda)$, we form the generating function
$$
F_{\tilde S} = \sum_{\hat S=(S,x)} F_{\hat S}
$$
where the sum is taken over all the labelled schemes $\hat S = (S,x)$ which are compatible with $\tilde S$.

The contribution of trees of scheme $\tilde S$ to~\eqref{eq:bigsum} is given by the coefficient:
\begin{align}\label{eq:Cauchy}
	[t^n] F_{\tilde S}(t) = \frac{1}{2\pi i}\oint_{\mathcal{I}_n} \frac{dt}{t^{n+1}} F_{\tilde S}(t),
\end{align}
		where $\mathcal{I}_n$ is the contour displayed on Figure~\ref{fig:hankel}-Left, which follows a circle or radius $r$ except close the positive real axis where it follows a small detour to the left of the singularity $1/4$ at distance $1/n$. Here $r>1/4$ is chosen so that $U(t)$ and $B(t)$ have no other singularity than $1/4$ inside the circle of radius $r$. We  split the countour $\mathcal{I}_n$ into the "circle part" $\mathcal{C}_n$ which is a portion of the circle of radius $r$, and the "Hankel part" $\mathcal{H}_n$.
We choose $r$ close enough to $\frac{1}{4}$ so that $\mathcal{H}_n$ is entirely contained in the neighbourhood $V$ previously chosen, and such that $U(t)<1$ for all $t\in \mathcal{I}_n$ (this is possible since by Pringsheim's theorem, we have $|U(t)|<1$ on the circle of radius $1/4$ from which $V$ is removed. By compactness and continuity, we can thus increase the radius a little bit and keep this inequality).

We note that $|B(1-U)|$ is bounded independently of $n$ on $\mathcal{C}_n$, and since $|U|\leq 1$, we have $|U^k|\leq 1$ for any $k=k(n)\geq 0$.
Hence, putting together~\eqref{eq:FShat},~\eqref{eq:WeInterne},~\eqref{eq:tableWeExterne},~\eqref{eq:WesumF2} we obtain that for $t \in \mathcal{C}_n$, 
		\begin{align}\label{eq:estimateFhat1}
			|F_{\hat S}(t)| \leq C_p \cdot  |B|^{m_{int}+m_1} \cdot  |U|^{\sum_{e \in F_{int}\cup F_{0}\cup F_{1}} |x_e^+-x_e^-|} \cdot  |1-U^q|^{m_{0} },
\end{align}
where
\begin{itemize}[itemsep=0pt, topsep=0pt,parsep=0pt, leftmargin=24pt]
	\item $F_{int}$ is the set of internal edges $e$ of $S$, $F_{0}$ is the set of non-frozen external edges such that $x^-_e \not\in[a+1,b]$, and $F_1$ is the set of frozen external edges;
	\item $m_{int}=|F_{int}|$ and $m_{0}=|F_{0}|$, $m_{1}=|F_{1}|$;
	\item for $e\in F_{0}$, we define $x^+_e$ to be equal to $a$ or to $b+1$ if $x^-_e<a$ or $x^-_e>b+1$, respectively;
	\item for $e\in F_{1}$, we fix arbitrarily $x^+_e\in\{a,a+1,b,b+1\}$ that minimizes $|x_e^+-x_e^-|$.
\end{itemize}

\medskip

We will now sum the estimate~\eqref{eq:estimateFhat1} over\footnote{Since we only work up to a constant multiplicative factor here, we will evaluate this sum approximately. An exact computation can be done following the lines of~\cite{CMS} but requires much heavier notation -- and this is ultimately not needed here.} all choices of abscissas $x_1,\dots,x_\ell$ in $\mathbb{Z}$ of the internal vertices of $S$ which are compatible with the scheme $\tilde S$.

In order to construct such a sequence $x_1,\dots,x_\ell$, we first choose a non-decreasing sequence
$$
y_1\leq \dots \leq y_{\ell+1}
$$
which \emph{in fine} will be the reordering of the sequence $0,x_1,\dots,x_\ell$. 

To construct a valid sequence $y_i$, we first choose which of the numbers $y_i$ are equal to zero (at least one is), and we choose which remaining indices will belong to the sets $I_0,I_1,I_2$. There are at most $C_p$ such choices. We then choose a value (among 4 possible) for each of the $n_2$ indices chosen to be in $I_2=\{a,a+1,b,b+1\}$, and there are again at most $C_p$ choices.
We then choose the values of each of the $n_1$ numbers inside $I_1=(a+2,b-1)$: we have at most $q^{n_1}$ choices. 

Finally we choose the $n_0$ remaining values,  assuming $n_0\geq 1$ (if $n_0=0$ there is no choice to make at this step). 

We call $y_{i_1}\leq \dots \leq y_{i_{n_0+1}}$ these $n_0$ numbers together with $0$, 
and we let $N:=y_{i_{n_0+1}}-y_{i_1}$. Assuming $N$ is fixed, there are at most $N^{n_0-1}$ choices for these numbers.
Moreover, applying the triangle inequality for edge increments along the path between internal vertices of labels $y_{i_{n_0+1}}$ and $y_{i_1}$ in $S$, we observe that 
$\sum_{e \in F_{int}} |x_e^+-x_e^-|\geq N$.

It remains to shuffle the sequence $y_1, \dots,\hat 0, \dots, y_{\ell+1}$ (one zero removed) to get the sequence $x_1\dots,x_\ell$. The number of shuffles is at most $C_p$.

From~\eqref{eq:estimateFhat1} we thus deduce (recall that $|U|\leq 1$ for $t\in \mathcal{C}_n$) that, for $t\in \mathcal{C}_n$, 
\begin{align}
	|F_{\tilde S}(t)| &\leq C_p
	\cdot  |B|^{m_{int}{+m_1}} \cdot  |1-U^q|^{m_{0}} \cdot q^{n_1}
	\sum_{N \in \mathbb{Z}}N^{n_0-1}  |U|^{N} \nonumber \\
	&\leq C_p |B|^{m_{int}{+m_1}} \cdot  |1-U^q|^{m_{0} } \cdot
	\mu^{n_1} n^{\frac{n_1}{4}}\frac{1}{(1-|U|)^{n_0}}   \nonumber \\
	&\leq C_p |B|^{m_{int}{+m_1}+n_0} \cdot  |1-U^q|^{m_{0} } \cdot
	\mu^{n_1} n^{\frac{n_1}{4}}\label{eq:last1},
\end{align}
if $n_0\geq 1$, and in fact the same bound holds if $n_0=0$ since the sum over $N$ in the first line is not present in that case.

\medskip 

With this estimate in hand we can now estimate the integral~\eqref{eq:Cauchy}. We start with the contribution of the Hankel part $\mathcal{H}_n$. 
We perform the change of variable 
$$
		t=\frac{1}{4}\left(1+\frac{\tau}{n}\right),
$$
		where $\tau$ now lives on the classical Hankel contour $\mathcal{H}$ represented on Figure~\ref{fig:hankel}-Right. We furthermore split $\mathcal{H}$ into $\mathcal{H}_-$ and $\mathcal{H}_+$, consisting of the parts for which $Re(\tau)\leq \log^2 n$ or $Re(\tau)>\log^2 n$, respectively\footnote{The contours we use and the way to split them are classical, see e.g. the proof of the transfer theorems in~\cite{Flajolet}.}.

		\medskip 
We first look at the contribution of $\mathcal{H}_-$. We have from~\eqref{eq:Puiseux}, expressed in the $\tau$ variable:
\begin{align}\label{eq:PuiseuxTau}
	U-1= c (\tfrac{\tau}{n})^{1/4} + O(\sqrt{\tfrac{\log^2 n}{n}}) 
	\ \ , \ \ 
	B= c' (\tfrac{\tau}{n})^{-1/4} + O(1)
	\ \ , \ \ 
	\frac{1}{t^{n+1}} = 4^n e^{-\tau ( 1 + O(\frac{\log^2 n}{n}))}
\end{align}
and, recalling that  $q=\mu n^{1/4}$ with $\mu\leq 1$,  
\begin{align}\label{eq:PuiseuxTau2}
	U^{q} & =\exp ( -c\mu (\tau^{1/4}+ O(\tfrac{\log n}{n^{1/4}})) ).
\end{align}
where all big-O's are uniform in all parameters.
Moreover, since $\Re(\tau^{1/4}) \geq 0$ and $c,\mu>0$ we have\footnote{If $\Re(x)\leq 0$ then $|e^x-1|=|e^x-e^0|=|\int_{[0,x]} e^tdt| \leq |x|$.}
$$
|1-U^q|=\left|e^{-c\mu (\tau^{1/4}+ O(\tfrac{\log n}{n^{1/4}}))} - 1 \right|
\leq c \mu \tau^{1/4}(1 + O(\tfrac{\log n}{n^{1/4}})) = O(\mu \tau^{1/4}).
$$
We thus get from~\eqref{eq:last1} that for $\tau \in \mathcal{H}_-$ we have 
\begin{align*}
	|F_{\tilde S}| \leq C_p 
	n^{\frac{m_{int}{+m_1}+n_0+n_1}{4}} \mu^{n_1+m_0} |\tau|^{\frac{m_0-m_{int}-n_0-2}{4}}.
\end{align*}
It follows that the contribution of $\mathcal{H}_-$ to~\eqref{eq:Cauchy} is bounded, in modulus, by
\begin{align}\label{eq:finalEstimate4}
	C_p n^{\frac{m_{int}{+m_1}+n_0+n_1}{4}} \mu^{n_1+m_0}\oint_{\mathcal{H}_-} (n^{-1} d\tau)4^n e^{-\tau} |\tau|^{C_p} 
	\leq C_p 4^n n^{\frac{m_{int}{+m_1}+n_0+n_1-4}{4}} \mu^{n_1+m_0},
\end{align}
since the remaining function of $\tau$ is integrable thanks to the exponential factor.

We now consider the contribution of the remaining contours. 
Note that for $\tau \in \mathcal{H}_+$ or for $t \in \mathcal{C}_n$, because of the factor $t^{-n-1}$, the integrand of the corresponding contour integral is dominated by $4^n$ by a superpolynomial factor ($\exp(-\Theta(\log^2 n))$ in the first case and $(4r)^{-n}$ in the second.
Therefore the contribution of these contours is more than polynomially smaller than the one of $\mathcal{H}_-$, and then the estimate~\eqref{eq:finalEstimate4} is valid for the whole contribution. We thus obtain:
\begin{align}\label{eq:last2}
	{|[t^n]F_{\tilde S}|} \leq C_p 4^n n^{\frac{m_{int}{+m_1}+n_0+n_1-4}{4}} \mu^{n_1+m_0}.
\end{align}

This ends the complex-analytic part of the proof.

\subsubsection{Exponent counting and dominant configurations}
 
To bound the r.h.s. of \eqref{eq:last2}, we will prove the following lemma:

\begin{lemma}\label{lemma"exponentCounting}
For any skeleton, in notation above, we have 
$$
 n^{\frac{m_{int}{+m_1}+n_0+n_1-4}{4}} \mu^{n_1+m_0}
 \leq {n^{\frac{p}{2}-\frac{3}{2}}
 \mu^{\lceil \frac{p}{2}\rceil}}$$
Moreover, in the case of $p$ even, the equality holds if and only if $S$ is a binary tree with $k=p$ leaves (i.e. $p$ external edges) in which the external edges are attached by pairs to $\frac{p}{2}$ vertices of the scheme, which all have abscissa in $[a,b+1]$.
\end{lemma}
\begin{proof}
Since $n$ goes to infinity and $\mu\leq 1$, to maximize the upper bound we have, roughly speaking, to look for the largest possible value of $m_{int}+m_1+n_0+n_1$ and the smallest possible value of $n_1+m_0$. However there might be some competition between these two factors so we have to be more precise. 

We write $\ell=\ell_{un}+\ell_{bin}$ where $\ell_{un}$ and $\ell_{bin}$ are respectively the number of non-root internal vertices of $S$ with $1$ and $2$ children. Since $S$ has $k$ leaves, we have $k= 1+\ell_{bin}+\xi$, where $\xi$ is equal to one if the root is binary in the scheme, and zero otherwise.
Let $\delta=p-k$, which is nonnegative since all leaves are decorated at least once.
We have, noting that $\ell=m_{int}=n_1+n_1+n_2$ and that $\ell=\ell_{un}+k-1-\xi=p+\ell_{un}-\xi-\delta-1$, 
$$
m_{int}{+m_1}+n_0+n_1-4=2l{+m_1}-n_2-4=2p-6+2\ell_{un}{+m_1}-2\xi-2\delta -n_2
\leq
2p-6+\ell_{un}-\delta-2\xi -n_2
$$
where for the inequality we used that {$\ell_{un}-\delta +m_1\leq 0$} (or equivalently, that $\ell_{un}+k+m_1\leq p$, which holds since by construction leaves and non-root unary vertices are necessarily decorated at least once, but the $m_1$ leaves incident to frozen edges are decorated at least once more).
Hence
\begin{align}\label{eq:ineq1}
 n^{\frac{m_{int}{+m_1}+n_0+n_1-4}{4}} \leq n^{p-\frac{3}{2}}n^{\frac{\ell_{un}-n_2}{4}}n^{-\frac{\xi}{2}}n^{-\frac{\delta}{4}}.
\end{align}
Now, write $n_2=n_2^{e}+n_2^i$, where $n_2^{e}$ and $n_2^i$ are respectively, among the vertices of the scheme contributing to $n_2$, the vertices which are attached to at least one external edge, and the ones which are not attached to an external edge. Write furthermore $n_2^e = n_2^{e,un}+n_2^{e,bi}$, separating the contribution to $n_2^e$ of unary and binary vertices of $S$.

The number of external edges of $S$ is equal to the number of leaves $k$, so there are $k-m_0$ external edges which are such that $x^-_e \in [a+1,b]$.
For these edges, one necessarily has $x(\alpha_e)\in [a,b+1]$ and therefore their attachment vertex contributes to the quantity $n_1+n_2$.
Moreover, each inner vertex can be attached to at most two external edges, which implies that
\begin{align}\label{eq:ineqPairs}
k-m_0 \leq 2n_1+2n_2^{e,bi}+n_2^{e,un}+2\xi,
\end{align}
hence
\begin{align}\label{eq:ineqPairs2}
k -2n_2^{e,bi} - n_2^{e,un} -2\xi \leq m_0 + 2n_1 \leq 2(m_0+n_1),
\end{align}
and 
$$n_1+m_0 \geq \frac{k}{2} -n_2^{e,bi} -\frac{n_2^{e,un}}{2} +\frac{\chi}{2}-\xi,
$$
{where $\chi$ is equal to $1$ if $k+n_2^{e,un}$ is odd and to zero otherwise.}
Therefore 
\begin{align}\label{eq:ineq2}
\mu^{n_1+m_0} \leq \mu^{\frac{k}{2}-n_2^{e,bi}-\frac{n_2^{e,un}}{2}+\frac{\chi}{2}-\xi} 
= \mu^{\frac{p}{2}-\frac{\delta}{2}-n_2^{ebi}-\frac{n_2^{e,un}}{2}+\frac{\chi}{2}-\xi} 
\leq \mu^{\frac{p}{2}+\frac{\chi}{2}} n^{\frac{\delta}{8}+\frac{n_2^{e,bi}+n_2^{e,un}/2}{4}}n^{\frac{\xi}{4}},
\end{align}
where we used that $q\geq 1$, hence $\mu^{-1}\leq n^{1/4}$.

From~\eqref{eq:ineq1} and~\eqref{eq:ineq2}, we get
\begin{align}\label{eq:ineq3}
 n^{\frac{m_{int}{+m_1}+n_0+n_1-4}{4}}  \mu^{n_1+m_0}
\leq
n^{p-\frac{3}{2}}\mu^{\frac{p}{2}+\frac{\chi}{2}} n^{-\frac{\delta}{8}} n^{\frac{\ell_{un}-n_2+n_2^{e,bi}+n_2^{e,un}/2}{4}}.
\end{align}
(Note that we used $n^{-\frac{\xi}{2}}n^{\frac{\xi}{4}}=n^{-\frac{\xi}{4}}\leq 1$).
Now we have
$
\ell_{un} \leq n_2^i + n_2^{e, un}.
$
Indeed, a non-root unary vertex of the scheme is necessarily decorated, so it only appears with a label in~$\{a,a+1,b,b+1\}$ so it contributes to $n_2$.

It follows that
$$
\ell_{un}-n_2+n_2^{e,bi}+n_2^{e,un}/2 \leq 
n_2^i+n_2^{e,un}-n_2+n_2^{e,bi}+ n_2^{e, un}/2  =  n_2^{e,un}/2.
$$

Therefore~\eqref{eq:ineq3} implies 
$$
 n^{\frac{m_{int}+m_1+n_0+n_1-4}{4}} \mu^{n_1+m_0}
 \leq n^{\frac{p}{2}-\frac{3}{2}}
 \mu^{ \frac{p}{2}}
 \mu^{\frac{\chi}{2}} n^{\frac{n_2^{e,un}-\delta}{8}}.$$
 Note that $\mu^{\frac{\chi}{2}}\leq 1$ and that $n_2^{e,un}\leq \delta$, since vertices counted by $n_2^{e,un}$ are unary, hence decorated.
  
  Therefore, if $p$ is even we have the wanted inequality. If $p$ is odd, then either $k+n_2^{e,un}$ is odd, in which case $\frac{p+\chi}{2}=\lceil \frac{p}{2}\rceil$ and we also have the wanted inequality,   or $k+n_2^{e,un}$ is even. 
  But this implies $k+n_2^{e,un} \neq p$, and since $k+n_2^{e,un} \leq p$ (by counting decorated vertices) this implies $k+n_2^{e,un}<p$, hence $\delta>n_2^{e,un}$.
  Therefore we have $n^{\frac{n_2^{e,un}-\delta}{8}}\leq n^{-\frac{1}{8}}\leq \mu^{\frac{1}{2}}$ (using again~$\mu^{-1}\leq n^{1/4}$) and the wanted bound holds in all cases.

The statement about the equality case for even $p$ is direct by requiring equality in all the inequalities used along the proof (we need $\delta=0$ from the factor $n^{-\frac{\delta}{8}}$, therefore $k=p$ {and $m_1=0$}; moreover, equality in~\eqref{eq:ineqPairs2} forces $m_0=0$. Finally, equality in~\eqref{eq:ineqPairs} then implies that the $p$ external edges are all such that $x_e^-\in[a+1,b]$ and that they are attached together by pairs).
\end{proof}

We get from the last lemma and from~\eqref{eq:last2} that
$
|[t^n]F_{\tilde S} |\leq  
C_p4^n n^{p-\frac{3}{2}} \mu^{\lceil\frac{p}{2}\rceil}.
$

Since the quantity~\eqref{eq:bigsum} is the sum of $[t^n]F_{\tilde S}$ over the \emph{finite} set all schemes $\tilde S$, we have finally obtained~\eqref{eq:toProve}. This ends the proof Lemma~\ref{lemma:momentBound}.

\subsection{End of the proof of tightness}

To end the proof of tightness, by the Kolmogorov criterion ~\cite[Theo. 12.3]{BIL68}, it is enough to apply Lemma~\ref{lemma:momentBound} with {$p=4$} and to prove that the sequence $(n^{-1/2}\Delta^n_0)_n$ is tight.  This will also prove the $\left(\frac{1}{2}-`e\right)$-H\"olderianity in Theorem~\ref{thm:tightness}.

To prove tightness of $(n^{-1/2}\Delta^n_0)_n$, we prove that the second moment is finite:
\begin{lemma}
	We have $\mathbf{E}\left[ ((n^{-1/2}\Delta^n_0 ) ^2 \right] =O(1)$.
\end{lemma}
\begin{proof}
	The proof can be done with the same approach as the one of Lemma~\ref{lemma:momentBound}, but is much simpler. In particular, all calculations can be done explicitly and by hand. Since $\Delta^n_0=M^n_1-M^n_0$,  
	the quantity $Cat(n) \mathbf{E}\left[ ((\Delta^n_0 ) ^2 \right]$ can be handled as in~\eqref{eq:bigsum}, and is equal to 
$$
	T_n(1,1)-T_n(1,0)-T_n(0,1)+T_n(0,0).
$$
	We thus have to count trees with two marked vertices, with labels either $1$ (with weight $+1$) or $0$ (with weight $-1$). In the dominant case, the corresponding skeleton is a fork \includegraphics[height=8mm]{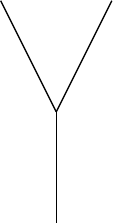}. Calling $x$ the label of the inner vertex of the fork, the corresponding generating function is thus $H_x (H_{x-1}-H_{x})^2$, and this expression can be summed over $x$ (to work with exact expressions, one should consider separately the cases $x<0$, $x=0$, $x=1$, $x>1$). After resummation and  examination of the singularity (the calculation is easily done explicitly\footnote{For example the case $x>1$ is evaluated as $\sum_{x\geq 1} BU^x (BU^{x-1}-BU^x)^2 = B^3 \frac{U(1-U)^2}{1-U^3}$. The behaviour close to $t=\frac{1}{4}$ is $O((1-4t)^{-1/2})$ as claimed.})
	the generating function is found to be of order $O((1-4t)^{-1/2})$ close to $t=\frac{1}{4}$.
	This singularity corresponds to a growth of $O(4^n n^{-1/2})$ for the coefficients (using either the transfer theorems of~\cite{Flajolet} or reproducing the argument of the previous section via Hankel contours). The subdominant case is handled similarly, and we finally get  
	$$Cat(n) \mathbf{E}\left[ (\Delta^n_0 ) ^2 \right] \leq 4^n n^{-1/2},
	$$
	which is what we need. 
\end{proof}

\subsection{Another consequence of the proof}
\label{sec:gaussianPairs}

Now that our estimate of~\eqref{eq:bigsum} is complete, we can go back and estimate precisely the first order asymptotic contribution. 
As a bonus, we will get the following proposition. The non-differentiability and non-Hölderianity in this statement give the second half of the conjecture in~\cite{BM-J} (the first half was proved in Theorem~\ref{thm:tightness}).

\begin{proposition}\label{prop:bonus}
Fix $\alpha\in \mathbb{R}$. We have the convergence in distribution
\begin{align}\label{eq:diffMoment}
 \frac{f'_{ISE}(\alpha+\mu)-f'_{ISE}(\alpha)}{\sqrt{\mu}}
 \xrightarrow[\mu\stackrel{>}{\rightarrow} 0]{(d)}
 \sqrt{2f_{ISE}(\alpha)} \mathcal{N},
\end{align}
where $\mathcal{N}$ is a standard centred and reduced Gaussian random variable independent from $f_{ISE}$.
Consequently, the function $f'_{ISE}$ is, almost surely, differentiable almost nowhere inside its support -- and in fact, it is not $(\frac{1}{2}+`e)$-Hölder-continuous, for any $`e>0$.
\end{proposition}

\begin{samepage}
Note that~\eqref{eq:diffMoment} can be seen as a (weak) discrete version of the heuristic
	$$
	d f_{ISE}(t)' \approx \sqrt{2f_{ISE}(t)} dB_t. \\*
	$$
	(with, informally speaking, $dt\approx \mu$ and $dB_t\approx \mu^{1/2}\mathcal{N}$). 
\end{samepage}

\begin{figure}
  \centerline{\includegraphics[width=0.65\linewidth]{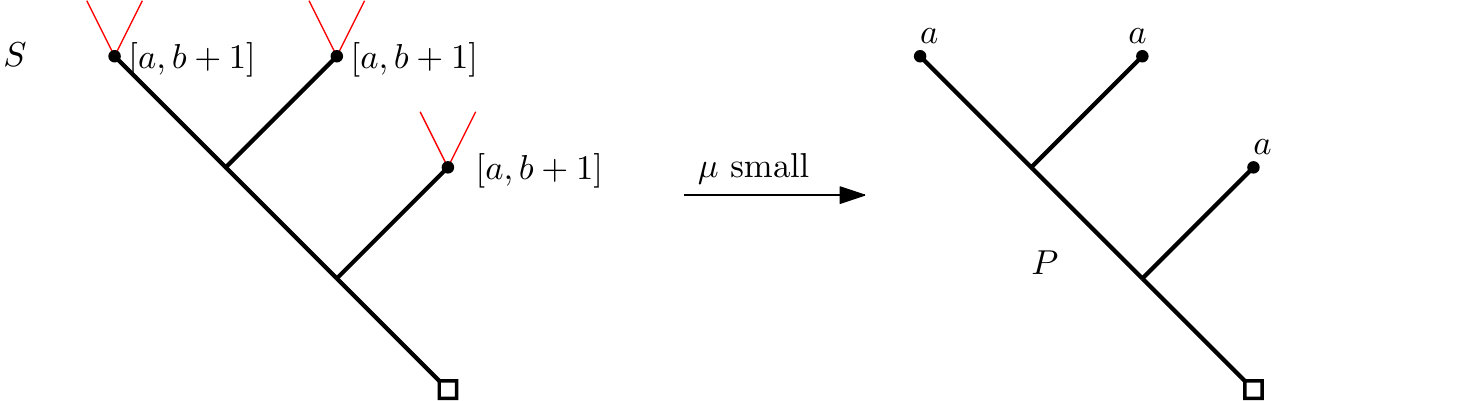}}
   \caption{Left: A scheme with $p=6$ leaves which dominates at first order the calculation of the moment $\mathbb{E}[(\Delta_{a}-\Delta_b)^{p}]$, after the proof of Lemma~\ref{lemma:momentBound} (the external edges, in red, are grouped in pairs and attached to vertices with abscissa in $[a,b+1]$). Calculations show that when $\mu=(b-a)n^{-1/4}$ goes to zero, the first-order contribution tends, up to identifiable factors, to the contribution that the "internal" scheme $P$ (in fat black) would give to the computation of the moment  $\mathbb{E}[(M_{a})^{p/2}]$, at the first order (Right). This observation leads to Proposition~\ref{prop:bonus}.
    }\label{fig:dominant}
  \end{figure}

\begin{proof}
We will prove convergence in law by proving convergence of moments, and we will first work at the discrete level.

{First consider the case of $p$ even, say $p=2r$.}
Notice that the proof of Lemma~\ref{lemma:momentBound}, and in particular Lemma~\ref{lemma"exponentCounting} shows that the schemes asymptotically dominating~\eqref{eq:bigsum}  
can be constructed as follows: start with a binary tree $P$ with $r$ leaves, attach a pair of dangling external edges to each leaf, and choose one of the $(2r-1)!!$ pairings of the leaves. Moreover the leaves of $P$ (internal vertices of $S$ to which external edges are attached) need to have abscissa in $[a,b+1]$. See Figure~\ref{fig:dominant}.

Let $S$ be such a scheme and let $\hat S=(S,x)$ be a compatible labelled skeleton. Call $z_1,\dots,z_{2r-1} \in [a,b+1]$ the labels of the vertices of $P$ (which are precisely the internal vertices of $S$), with $z_{r},\dots,z_{2r-1} \in [a,b+1]$ being the labels of the leaves of $P$.

By~\eqref{eq:FShat}, \eqref{eq:WeInterne} and~\eqref{eq:tableWeExterne}, the generating function corresponding to the labelled skeleton $\hat S$ in the computation of~\eqref{eq:bigsum} is given by
$$
F_{\hat S}(t) = t^{2r-1} T^{2r}
B^{|E(P)|} U^{\sum_{e \in E(P)} |x_e^- - x^+_e|} \times 
\prod_{i=1}^{2r} (B(1-U) (U^{b-z_i}+U^{z_i-a-1})^2,
$$
where we separated the contribution of internal and external edges ($E(P)$ denotes the set of edges of $P$, i.e. internal edges of $S$).
Note that for $t\in \mathcal{H}_n$, we have $B(1-U)=2+o(1)$ and\footnote{the term $o(1)$ is relative to $n$ going to infinity.} if $z_i\in [a,b+1]$ 
$$
U^{z_i-b+1}+U^{z_i-a} =  2 + O(\mu \tau^{1/4}) + o(1)
$$
as in~\eqref{eq:PuiseuxTau2}, with $\mu = (a-b)n^{-1/4}$ as before.
These estimates hold uniformly in $t \in \mathcal{H}_n$, in the $z_i$, and in $\mu$. Moreover, let $e$ be an external edge of $P$, we have $x^+_e\in [a,b+1]$ and
$$
U^{|x_e^- - x^+_e|}= U^{|x_e^- - a|}(1+O(\mu)+o(1)),
$$
again uniformly in all parameters. 
We deduce that, again uniformly,
$$
F_{\hat S}(t) = 4^r(1+O(\mu)+o(1)) t^{2r-1} T^{2r}
B^{|E(P)|} U^{\sum_{e \in E(P)} |z_e^- - z^+_e|}
$$
where the integers  $z^-_e,z^+_e$ are defined as $x_e^-, x_e^+$, but starting from the labelling $z_1,\dots,z_{r-1},a,a,\dots,a$ (in other words we fix the leaves of $P$ to abscissa $a$, but keep the abscissas of internal vertices).
Using that $t$ and $T$ go respectively to $\tfrac{1}{4}$ and $1$ on $\mathcal{H}_n$, we can also write
$$
F_{\hat S}(t) = (1+O(\mu)+o(1)) t^{r-1} T^{r}
B^{|E(P)|} U^{\sum_{e \in E(P)} |z_e^- - z^+_e|}.
$$
Since all big and little-Os are uniform,  and since as before contributions of the contour integrals outside of $\mathcal{H}_n$ can be neglected in~\eqref{eq:Cauchy} we conclude that
\begin{align}
\label{eq:inte1}
[t^n]F_{\hat S}
=
(1+O(\mu)+o(1)) 
[t^n] t^{r-1} T^{r}B^{|E(P)|} U^{\sum_{e \in E(P)} |z_e^- - z^+_e|} 
\end{align}
 (note that the exponential term $t^{-n-1}$ gives dominated convergence justifying all the estimates).
Now, the quantity  $t^{r-1} T^{r} B^{|E(P)|} U^{\sum_{e \in E(P)} |z_e^- - z^+_e|}$ is easily recognized, by~\eqref{eq:WeInterne} and reasoning similarly as in Section~\ref{subsubsec:branches}, as the generating function of binary trees with $r$ marked vertices with skeleton $P$, where marked vertices have all abscissas $a$, and such that the internal vertices of $P$ have labels $z_1,\dots,z_{r-1}$.

Since all estimates are uniform, we can now sum~\eqref{eq:inte1} over $z_1,\dots,z_{2r-1}$ and over all binary trees $P$.
We deduce that~\eqref{eq:bigsum} is equal to
$$
 (2r-1)!! 2^r (1+O(\mu)+o(1)) q^r T^*_n(\underbrace{a,a,\dots,a}_{r\mbox{ \footnotesize times}}),
$$
where $T^*_n(i_1,\dots,i_r)$ is the contribution to the number $T_n(i_1,\dots,i_r)$ of marked trees whose $r$ marked vertices are in generic position (i.e. their skeleton is a binary tree with $r$ leaves). 
To obtain this expression, we have summed first  over $z_r,\dots,z_{2r-1} \in [a,b+1]$ giving the factor $q^r$ (note that these variables do not appear in~\eqref{eq:inte1}), while the sum over $P$ and over its internal labels $z_1,\dots,z_{r-1}$ accounts for the possible relative positions of the $r$ marked vertices of label $a$ and of the abscissas of the internal vertices of their skeleton, in a configuration counted by $T^*_n(a,a,\dots,a)$.
Note also  the combinatorial factor $(2r-1)!!2^r=\frac{(2r)!}{r!}$ which is the ratio between the possible numberings of the $2r$ leaves of $S$ and the $r$ leaves of $P$.
Now, it is easily seen using again the same scheme techniques that for $a=O(n^{1/4}),$ we have 
$T^*_n(a,a,\dots,a) 
\sim T_n(a,a,\dots,a) $.

Putting everything together, we finally obtain that, with $a,b=O(n^{1/4})$
$$
Cat(n) \mathbb{E} \big[(\Delta^n_{a+1}-\Delta^n_{b+1})^{2r}\big]
=   (2r-1)!! 2^r (1+O(\mu)+o(1)) n^{\frac{r}{4}} \mu^r T_n(a,a,\dots,a),
$$
or equivalently
\begin{align}\label{eq:moments3}
\mathbb{E} \left[\left(\frac{\Delta^n_{a+1}-\Delta^n_{b+1}}{\sqrt{n}}\right)^{2r}\right] = 
(2r-1)!!
 \mu^r \mathbb{E} \left[\left(\frac{2 M^n_a}{n^{3/4}}\right)^r\right]  (1+O(\mu)+o(1)) .
\end{align}

Taking the limit $n\rightarrow \infty$ on both sides with $a=\alpha n^{1/4}$ and $b=(\alpha+\mu)n^{1/4}$, we obtain,
\begin{align*}
\mathbb{E}\left[ 
\left(f'_{ISE}(\alpha+\mu)-f'_{ISE}(\alpha)\right)^{2r}\right]
&= 
(2r-1)!! 
\mu^r
\mathbb{E}\left[ 
\left(2 f_{ISE}(\alpha) \right)^r
\right] (1+O(\mu)) \\
&= \mu^r
\mathbb{E}\left[ 
\left(2f_{ISE}(\alpha)\right)^r \mathcal{N}^{2r}
\right] (1+O(\mu)) ,
\end{align*}
where $\mathcal{N}$ is a standard Gaussian.
{To deduce the first equality,  we have used the convergence in law of Theorem~\ref{thm:tightness}, together with the fact that, by Lemma~\ref{lemma:momentBound}, the moments appearing on both sides of~\eqref{eq:moments3} are bounded independently of $n$ for any $r>0$ -- these two facts imply the convergence of moments.}
This further implies 
\begin{equation}\label{eq:momentsConv}
\mathbb{E}\left[ 
\left(\frac{f'_{ISE}(\alpha+\mu)-f'_{ISE}(\alpha)}{\sqrt{\mu}}\right)^{p}\right]
\stackrel{\mu\rightarrow 0}{\longrightarrow}
\mathbb{E}\left[ 
\left(2f_{ISE}(\alpha) \right)^{\frac{p}{2}} \mathcal{N}^{p}.
\right].
\end{equation}

{Now consider the case of $p$ odd. By Lemma~\ref{lemma:momentBound}, we directly have
\begin{align*}
\left|
\mathbb{E} \left[\left(\frac{\Delta^n_{a+1}-\Delta^n_{b+1}}{\sqrt{n}}\right)^{p}\right]
\right|
\leq C_p \mu^{\frac{p}{2}+\frac{1}{2}}.
\end{align*}
Using the convergence of moments justified above, we deduce by taking the limit $n\rightarrow \infty$,
\begin{equation}
\left|
\mathbb{E}\left[ 
\left(\frac{f'_{ISE}(\alpha+\mu)-f'_{ISE}(\alpha)}{\sqrt{\mu}}\right)^{p}\right]
\right|
\leq C_p \sqrt{\mu},
\end{equation}
and we deduce by taking the limit $\mu\rightarrow 0$ that~\eqref{eq:momentsConv} also holds for odd $p$ (in that case, the right-hand-side is null, since the Gaussian variable has null odd moments).
}
This implies  the convergence in distribution~\eqref{eq:diffMoment} (note that moments of $f_{ISE}(\alpha)$ do not grow too fast, see e.g.~\cite{BM}).

Now take $\alpha=0$, since $f_{ISE}(0)>0$ almost surely, this implies that $f'_{ISE}$ is, almost surely, not differentiable at $0$. By rerooting invariance, this implies that $f'_{ISE}$ is, almost surely, differentiable almost nowhere inside of its support.
The statement about non-Hölderianity is similar.
\end{proof}

\begin{remark}
It is possible, at the cost of heavier notation but with the same tools and without new significant  difficulty, to enrich the counting techniques developed throughout Section~\ref{sec:proofTight} to estimate joint moments of the form
\begin{align}\label{eq:jointJoint}
\mathbb{E}\left[
\prod_{i=1}^A (\Delta^n_{a_i}-\Delta^n_{b_i})^{r_i}
(\Delta^n_{a_i})^{s_i}
(M^n_{a_i})^{t_i},
\right]
\end{align}
{for integers numbers $a_i,b_i$ and $r_i,s_i,t_i\geq0$}.
To do this, one only has to consider more general schemes in which the marked vertices can be of three types (corresponding to the three types of factors in~\eqref{eq:jointJoint}). The generating function $W_e$ corresponding to each type of edge can be computed as before. At the asymptotic level, the same phenomenon will appear and the dominating contributions are the one in which for each $i\in [A]$, the $r_i$ vertices of the first type share their attachment points in pairs. When furthermore $a_i-b_i=\mu_i n^{1/4}$ with $\mu_i$ small, each attachment vertex plays the same role as a vertex of label $a_i$ up to easily identified factors. In this way, one can prove, {for even $r_i$}:
\begin{align*}
\mathbb{E}\left[
\prod_{i=1}^A \left(\frac{\Delta^n_{a_i}-\Delta^n_{b_i}}{\sqrt{{\mu_i} n}}\right)^{r_i}
\left(\frac{\Delta^n_{a_i}}{n^{3/8}}\right)^{s_i}
\left(\frac{M^n_{a_i}}{n^{3/4}}\right)^{t_i}
\right]&=
\mathbb{E}\left[
\prod_{i=1}^A \left(\sqrt{\frac{2 M^n_{a_i}}{n^{3/4}}} \mathcal{N}_i\right)^{r_i}
\left(\frac{\Delta^n_{a_i}}{n^{3/8}}\right)^{s_i}
\left(\frac{M^n_{a_i}}{n^{3/4}}\right)^{t_i}\right]\\
&\times (1+o(1)+O(\|\mu\|)),
\end{align*}
when $a_i,b_i = O(n^{1/4})$ with $a_i-b_i=\mu_i n^{1/4}$, and where the $\mathcal{N}_i$ are standard independent Gaussian random variables (centred, and having variance 1), independent of everything else. {If one $r_i$ is odd, the LHS is a $O(\sqrt{\mu_i})$}.

It follows, by taking the limit when $n$ goes to infinity, that the convergence of Proposition~\ref{prop:bonus} can be strengthened as follows. For any $\alpha_1,\dots,\alpha_A$ we have the convergence in law of vectors
\ben
\nonumber&&\left[ \l(f_{ISE}'(\alpha_i), f_{ISE}(\alpha_i), \frac{f'_{ISE}(\alpha_i+\mu_i)-f'_{ISE}(\alpha_i)}{\sqrt{\mu_i}  }\r), 1\leq i \leq A \right] ~~~~~~~~~~~~~~~~~~~~\\
&&~~~~~~~~~~~~~~~~~~~~\xrightarrow[\mu_i \to 0^+, \forall i]{(d)}
\Big[ \l(f_{ISE}'(\alpha_i), f_{ISE}(\alpha_i),\sqrt{2f_{ISE}(\alpha_i)}\mathcal{N}_i, 1\leq i \leq A \r)\Big],
\een
where in the right hand side, the $\mathcal{N}_i$ are i.i.d. ${\cal N}(0,1)$ Gaussian random variables, independent from $\l[\l(f_{ISE}'(\alpha_i), f_{ISE}(\alpha_i), 1\leq i \leq A \r) \r]$.
Setting up all the notation for a formal proof would go beyond the intent of this note, and we hope that more direct diffusion approximation techniques might give another approach to such results.
\end{remark}

\color{black}

\section{Proof of the diffusion approximation}
We mainly work with $\xi^{n,\star}=(\delta^{n,\star},m^{n,\star},s^{n,\star})$ and use the stuck version (when $m^{n,\star}=`e$) only when needed.
 We denote by ${\cal B}_k^n$ the $\sigma$-algebra generated by the $\l(\xikn{i}, i \leq k\r)$ (which includes $\l( \Delta_0^\star, M^\star_0,S_0^\star\r)$).
 
 For all $i$, set $d\tkn{i} = n^{-1/4}$ (this is the homogeneous time increment).
 We have
\ben\label{eq:gerdq1}
\E_{~\Bkn}\l(\xikn{k+1}- \xikn{k} \r)   & = & f_n\l(\xikn{k},\tkn{k}\r)\, d \tkn{k} \\ 
\label{eq:gerdq2}\Cov_{~\Bkn}\l(\xikn{k+1}- \xikn{k} \r) & = & \sigmakn{k}\l(\xikn{k},\tkn{k}\r) \,\sigmakn{k}^{T}\l(\xikn{k},\tkn{k}\r)\, d\tkn{k}
\een
where $T$ denotes the transpose, with
\[f_n\l(\bma d\\m\\s\ema,t\r)=\bma 0\\ d \\ m\ema,~~\sigma_n\l(\bma d\\m\\s\ema,t\r)=\sqrt{2|m|}\bma 1\\n^{-1/4}\\n^{-1/2}\ema \]
(they are homogeneous, so that $f_n$ as well as $\sigma_n$ do not depends on $t$).
Formulas \eref{eq:gerdq1} and \eref{eq:gerdq2} come from the following simple facts:
\ben
\bpar{ccl}
\Var_{\Bkn}\l(\Delta_{k+1}^{\star}-\Delta_k^\star \r)&=& 2|M_{k}^\star| \\
\Var_{\Bkn}\l(M_{k+1}^\star-M_k^\star \r)&=& \Var_{\Bkn}\l( \Delta_{k+1}^\star \r)=\Var_{\Bkn}\l( \Delta_{k+1}^\star-\Delta_{k}^\star \r)=2|M_{k}^\star|\\
\Var_{\Bkn}\l(S_{k+1}^\star-S_k^\star \r)&=&\Var_{\Bkn}\l( M_{k+1}^\star\r)=\Var_{\Bkn}\l(M_{k+1}^\star-M_{k}^\star \r)=2|M_k^\star|\\
\Cov_{\Bkn}\l(M_{k+1}^\star-M_k^\star,\Delta_{k+1}^\star-\Delta_k^\star\r)&=&\Cov_{\Bkn}\l(\Delta_{k+1}^\star,\Delta_{k+1}^\star\r)=2|M_k^\star| \\
\Cov_{\Bkn}\l(S_{k+1}^\star-S_k^\star,\Delta_{k+1}^\star-\Delta_k^\star\r)&=&\Cov_{\Bkn}\l(M_{k+1}^\star,\Delta_{k+1}^\star-\Delta_k^\star\r)=2|M_k^\star| \\
\Cov_{\Bkn}\l(S_{k+1}^\star-S_k^\star,M_{k+1}^\star-M_k^\star\r)&=& \Cov_{\Bkn}\l(M_{k+1}^\star,M_{k+1}^\star-M_k^\star\r)=2|M_k^\star|;\epar
\een
indeed, the variance of the geometric random variables involved is 2; we use also that if $Y$ is ${\cal F}$-measurable,$\Cov_{\cal F}(X,Z)=\Cov_{\cal F}(X-Y,Z)=\Cov_{\cal F}(X-Y,Z-Y)$.
The (non rescaled) covariance matrix of $Z^\star_{k+1}-Z^{\star}_k$ is then
\[2|M_k^\star| \bma 1 & 1 & 1  \\1 & 1 & 1  \\1 & 1 & 1  \ema= 2|M_k^\star| \bma 1 \\1 \\1 \ema \bma 1 & 1 & 1 \ema,\]
so that if one writes $\xikn{k}(i)$ for the $i$th entry of $\xikn{k}$, the covariance matrix is
\be
A_{\Bkn}&=&2|M_k^\star| \bma 1/n^{(i+j +2 )/4}\ema_{1\leq i,j \leq 3}
= \frac{2|M_k^\star|}{n^{3/4}} \bma 1/n^{(i+j -2 )/4}\ema_{1\leq i,j \leq 3}dt_k^n. 
\ee

We can now prove the diffusion approximation.

\begin{proof}[Proof of Theorem~\ref{thm:diffApprox}]
We rely on the main theorem in Kushner \cite{Kushner}, and we keep the notation of this reference.
In \cite{Kushner}, there are 7 conditions to check (1) and (A1) to (A6).

 Let us  
introduce the good set
\[GS(\varepsilon,C)= [-C,C] \times [`e/2  , C ]\times [0,C]\]
for some $C>m_0>`e$.
The drift and coefficient functions $f$ and $\sigma$ are bounded and Lipschitz on $GS(`e,C)$, and this will give a sufficient condition for the existence and uniqueness of a solution of the SDE \eref{eq:qfge} (see Condition (A6) below).

It suffices to check that $\xi^{n,\star}$  and $\xi^{\star}$ satisfies the 7 Kushner constraints (on $GS(\varepsilon,C)$, and that it does so, for all $C>0$):

\smallskip

\noindent \textbf{Condition (1)}
First, we have $\sum_{k=0}^{N_n-1}|f_n(\xikn{k},\tkn{k})-f(\xikn{k},\tkn{k})|^2  d\tkn{k}=0 \to 0,$ when $\xi^{(n),\star}$ stays in $GS(\varepsilon,C)$.
Second, 
\[\sigma_n(\xikn{k},\tkn{k})-\sigma(\xikn{k},\tkn{k})=2\sqrt{|\xikn{k}(2)|}\bma 0 \\n^{-1/4}\\n^{-1/2}\ema.\]
We then need to prove that $\E(\sum_{k=0}^{N_n-1}\|\sigma_n(\xikn{k},\tkn{k})-\sigma(\xikn{k},\tkn{k})\|^2 ) d\tkn{k} \to 0$ which is equivalent to 
\[\l(\sum_{k=0}^{N_n-1}  \E(2|M_k^\star|n^{-3/4}) n^{-1/4}\r) n^{-1/4}\to 0. \]
Now, on  $GS(\varepsilon,C)$, $|M_k^\star|\leq Cn^{3/4}$, so the result follows immediately.

\noindent \textbf{Condition (A1)} $\max_{0\leq k \leq N_n-1} d\tkn{k}=n^{-1/4}\to 0$ when $n\to +\infty$.

\noindent \textbf{Condition (A2)} $f$ and $\sigma$ are indeed continuous and bounded on $GS(\varepsilon,C)$, and $f_n$ and $\sigma_n$ are uniformly bounded  on $GS(\varepsilon,C)$.

\noindent \textbf{Condition (A3)} This condition concerns the convergence of the initial distribution of $\xi^{n,\star}(0)$ (which is one of the hypotheses).

\noindent \textbf{Condition (A4)} Set 
\[B_n:=\E\l(\sum_{k=0}^{N_n-1}\|\xikn{k+1}-\xikn{k}-f_n(\xikn{k},\tkn{k})d \tkn{k}\|^{2+\alpha}\r).\]
In fact, only the first entry of the vector $\xikn{k+1}-\xikn{k}-f_n(\xikn{k},\tkn{k})d\tkn{k}$ is not zero:
One has, since $|M_k^\star|\leq C n^{3/4}$,
\be
B_n&=& \sum_{k=0}^{N_n-1}\E\l(|W_{|M_k^\star|}n^{-1/2}|^{1+\alpha/2}\r)\\
&\leq& n^{1/4} \max_{ 0\leq m \leq Cn^{3/4}} \E\l(|W_{|m|}n^{-1/2}|^{1+\alpha/2}\r)\\
& =& n^{1/4} \max_{ 0\leq m \leq Cn^{3/4}} \E\l(|W_{|m|}|^{1+\alpha/2}\r)n^{-1/2-\alpha/4}.\ee
We need to prove that for a well chosen $\alpha>0$, $B_n\to 0$.
Now, since the increments of $W$ have all finite moments, by Marcinkiewicz-Zygmund inequality
  \[\E\l(|W_{|n|}|^{1+\alpha/2}\r)\leq  B(2\, C\,n^{3/4})^{(1/2+\alpha/4)}\]
  where the $2$ comes from the variance of $g^{(k)} -1$, and $B$ is a positive function of $\alpha$.
It suffices then to take $\alpha$ such that
\[1/4 +(3/4)(1/2+\alpha/4)   -1/2-\alpha/4 =1/8-\alpha/16<0\]
and any $\alpha>2$ does the job.

\noindent \textbf{Condition (A5)} $d\tkn{k+1}/ d\tkn{k}=1$. Clear

\noindent \textbf{Condition (A6)}  Here, $f(.,.)$ is Lipschitz as well as $\sigma$ on $GS(\varepsilon,C)$ except at its boundary.
These conditions are sufficient to entail the existence and uniqueness of the solution of the SDE \eref{eq:qfge}, started at $z_0$ (see \O ksendal \cite[Theorem 5.2.1]{OKS}), stuck when $m^{\star}=`e$.
Notice that the value $`e/2$ in the second component in $GS(`e,C)$ is taken smaller than $`e$, which is the point at which we stuck $m_n^{\star}$ and $m^\star$. The value $`e$ being at the interior of $[`e/2,C]$, one sees that the domain on which one can extend the existence of the solution of the SDE is sufficient to entail the convergence of the stuck version of $\zeta^{n,\star}$ to the stuck version of $\zeta^\star$.

\end{proof}

\bibliographystyle{abbrv}

\end{document}